\documentclass[12pt,fullpage]{article}

\usepackage{amssymb}
\usepackage{amsmath}
\usepackage{amsthm}
\usepackage{dsfont}
\usepackage{epsfig}
\usepackage{multicol}
\usepackage{caption}
\usepackage{psfrag}
\usepackage{xcolor}
\usepackage{graphicx}
\usepackage{authblk}
\input epsf

\numberwithin{equation}{section}

\newtheorem{Lemma}{Lemma}[section]
\newtheorem{Cor}[Lemma]{Corollary}
\newtheorem{Th}[Lemma]{Theorem}
\newtheorem{Prop}[Lemma]{Proposition}
\theoremstyle{definition}
\newtheorem{Def}[Lemma]{Definition}
\theoremstyle{remark}
\newtheorem{Rk}[Lemma]{Remark}
\newtheorem{Rks}[Lemma]{Remarks}

\newcommand{\be}{\begin{equation}}
\newcommand{\ee}{\end{equation}}
\newcommand{\baa}{\begin{array}}
\newcommand{\eaa}{\end{array}}
\newcommand{\ba}{\begin{eqnarray}}
\newcommand{\ea}{\end{eqnarray}}
\newcommand{\ds}{\displaystyle}
\renewcommand\phi{\varphi}

\def\he{h^\star}
\def\ph{\phi_h}
\def\pe{\phi_{\he}}
\def\psihe{\psi_{h, e}}

\def\epsilon{\varepsilon}
\def\O{\Omega}
\def\Rm{\R^{N-1}}
\def\pn{\partial_\nu}
\def\pO{\partial \O}

\def\qtext#1{\quad\text{#1}\quad}
\def\pref#1{{\rm (\ref{#1})}}

\def\S{{\mathbb S}}
\def\Sm{\S^{N-2}}

\def\R{{\mathbb R}}
\def\N{{\mathbb N}}

\def\L{{\mathcal L}}

\def\trait (#1) (#2) (#3){\vrule width #1pt height #2pt depth #3pt}
\def\fin{\hfill\trait (0.1) (5) (0) \trait (5) (0.1) (0) \kern-5pt \trait (5) (5) (-4.9) \trait (0.1) (5) (0)}
\def\todo#1 {\marginpar{\textcolor{red}{$\Rightarrow$#1}}}

\newcommand{\bee}{\begin{equation*}}
\newcommand{\eee}{\end{equation*}}
\newcommand{\bc}{\begin{cases}}
\newcommand{\ec}{\end{cases}}

\title{\bf Front blocking and propagation in cylinders with varying cross section}
\author[1]{Henri Berestycki}
\author[2]{Juliette Bouhours}
\author[1,3]{Guillemette Chapuisat}
\affil[1]{Ecole des Hautes Etudes en Sciences Sociales et CNRS, Centre d'Analyse et de Math\'ematiques Sociales, UMR 8557, 190-198, avenue de France 75244 Paris Cedex 13, France}
\affil[2]{Sorbonne Universit\'es, UPMC Univ Paris 06, et CNRS, Laboratoire Jacques-Louis Lions, UMR 7598,F-75005, Paris, France}
\affil[3]{Aix-Marseille Universit\'e, CNRS, Centrale Marseille, Institut de Math\'ematique de Marseille, UMR 7373, 13453 Marseille, France}

\begin{document}
\maketitle

\maketitle

\begin{abstract}
In this paper we consider a bistable reaction-diffusion equation in unbounded domains and we investigate the existence of propagation phenomena, possibly partial, in some direction or, on the contrary, of blocking phenomena. We start by proving the well-posedness of the problem. Then we prove that when the domain has a decreasing cross section with respect to the direction of propagation there is complete propagation. Further, we prove that the wave can be blocked as it comes to an abrupt geometry change. Finally we discuss various general geometrical properties that ensure either partial or complete invasion by 1. In particular, we show that in a domain that is ``star-shaped" with respect to an axis, there is complete invasion by 1.
\end{abstract}

%{$\ $}\vspace{-50pt}

%%%%%%%%%%%%%%%%%%%%%%%%%%%%%%%%%%%%%%%%
%%%%%%%%%%%%%%%%%%%%%%%%%%%%%%%%%%%%%%%%
\tableofcontents

\vskip 20pt

\section{Introduction}\label{Intro}
\subsection{The problem}
In this paper we consider the following parabolic problem
\be\label{problem}
\begin{cases}
\partial_tu(t,x)-\Delta u(t,x)=f(u(t,x)), &\textrm{for } t\in\R,\quad x\in\Omega,\\
\partial_\nu u(t,x) =0, &\textrm{for } t\in\R,\quad x\in\partial\Omega,
\end{cases}
\ee
where $\Omega$ is an unbounded domain such that 
\be\label{Omegadef}
\Omega=\left\{(x_1,x'), x_1\in\R, x'\in\omega(x_1)\subset\R^{N-1}\right\}.
\ee
Throughout the paper, we assume $\omega$ not to depend on $x_1$ for $x_1<0$ and $f$ is a bistable nonlinearity. We call such domains ``cylinder-like". We give more precisions on the general assumptions in the next section but one can see in Figure \ref{exomega} various examples of the type of domains $\Omega$ we have in mind in dimension 2.
\begin{figure}[h!]
\begin{center}
\includegraphics[scale=0.6]{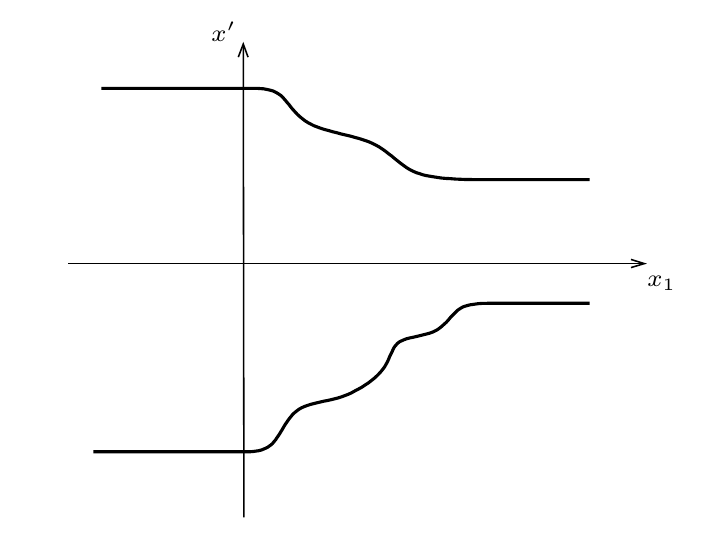}
\includegraphics[scale=0.5]{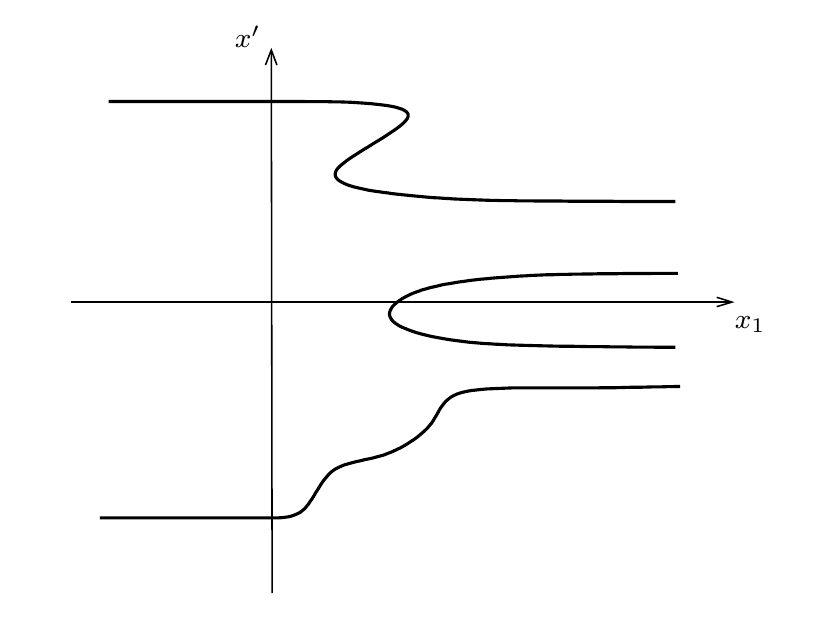}
\includegraphics[scale=0.35]{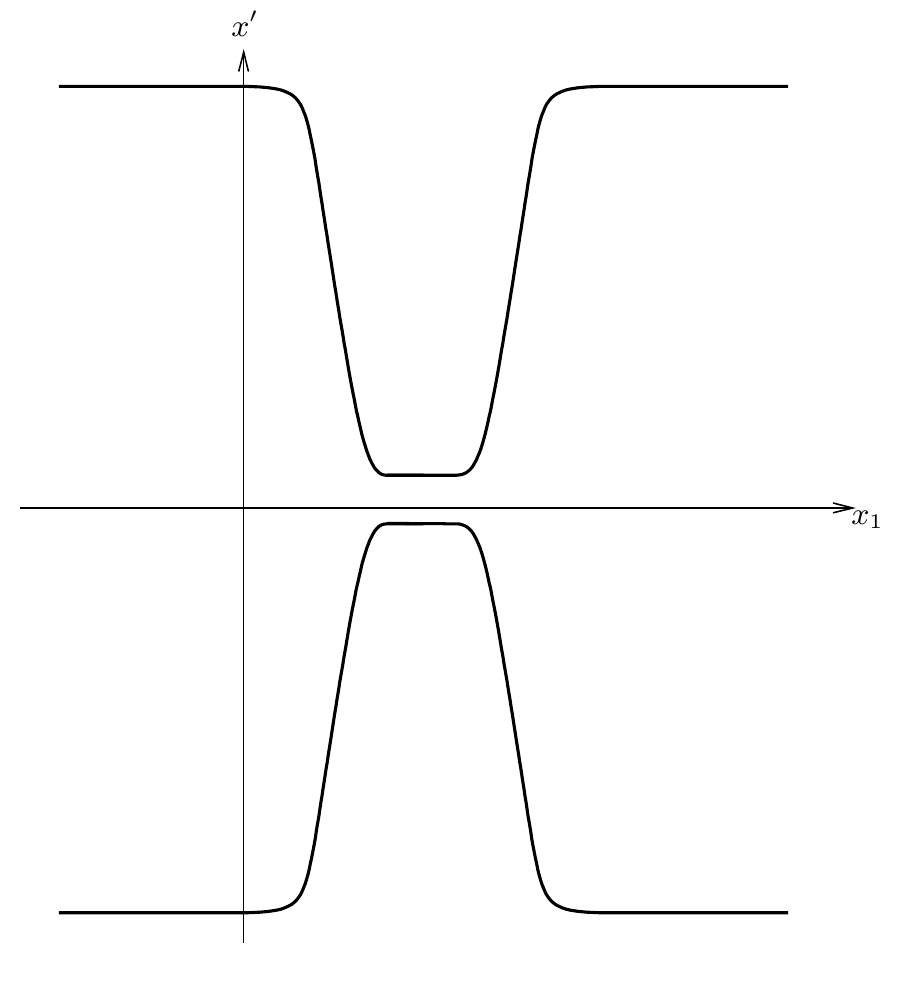}
\includegraphics[scale=0.75]{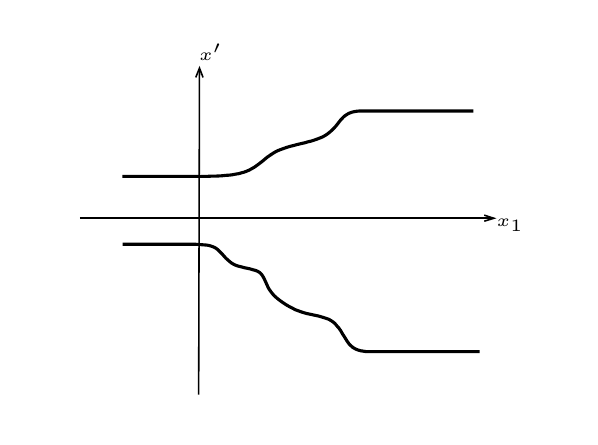}\quad
\includegraphics[scale=0.7]{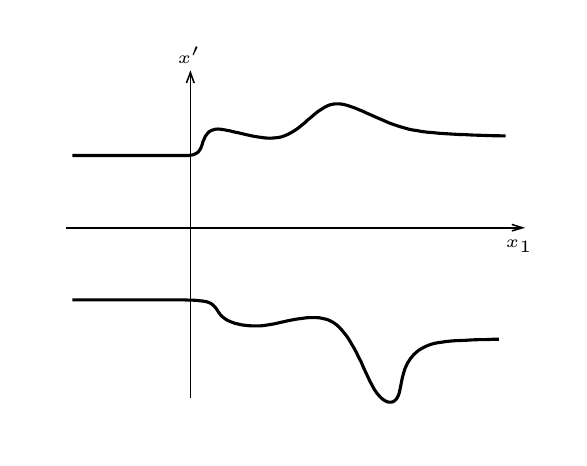}\quad
\includegraphics[scale=0.55]{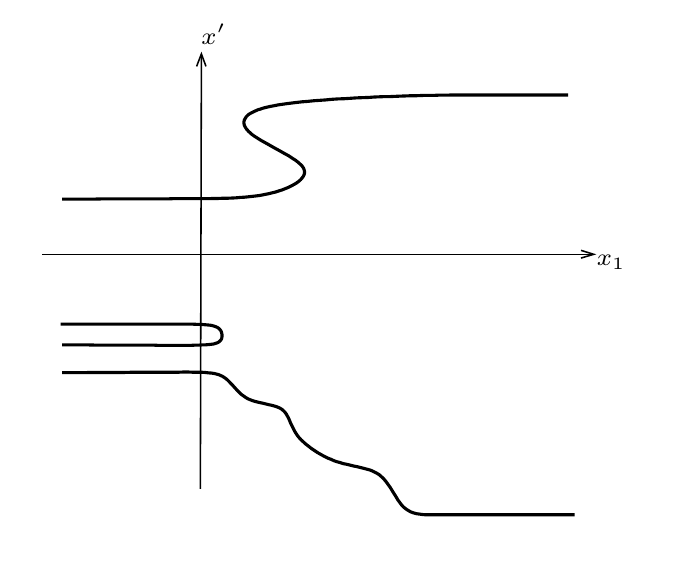}

\caption{Examples of ``cylinder-like" domains $\Omega$. Propagation is considered in the direction of increasing $x_1$.}\label{exomega}
\end{center}
\end{figure}

The aim of this paper is to describe invasions coming from $-\infty$. That is, waves propagating from $x_1=-\infty$ in the direction of increasing $x_1$. These waves have implications for the initial value problem for initial data with, say, compact support. We can also envision domains with two different axes, with the direction of propagation changing as $t\to-\infty$ and as $t\to+\infty$. Think for instance of ``bent cylinders". But for simplicity, in this paper we only present the results in the framework of a single axis.
The main results in this paper are the following.
\begin{itemize}
\item We first study the existence of transition fronts in this setting, i.e the existence and uniqueness of an entire solution $u$ of the parabolic problem \eqref{problem}, such that 
\be\label{condinf}
u(t,x)-\phi(x_1-ct)\to0 \textrm{ as } t\to-\infty \textrm{ uniformly in } \overline{\Omega},
\ee
where $(\phi,c)$ is the unique bistable travelling front solution of:
\be\label{tw}
\begin{cases}
\phi''+c\phi'+f(\phi)=0 \quad\textrm{in } \R,\\
\phi(-\infty)=1, \hspace{0.1cm}\phi(+\infty)=0,\hspace{0.1cm} \phi(0)=\theta.
\end{cases}
\ee
\item We establish the {\em blocking} of the solution if there exists $a\in\R$ such that the measure of $\Omega\cap\{a<x_1<a+1\}$ is small,
\item We prove that there is {\em axial partial propagation} when $\Omega$ contains a straight cylinder in the $x_1$-direction, whose cross section is a ball of sufficiently large radius.
\item We give conditions on the geometry of $\Omega$ under which there is {\em complete invasion} that is, $u(t,x) \to 1$ as $t\to\infty$ at every point $x$:
\begin{itemize}
\item[$\cdot$] In the case of general domains with decreasing cross section,
\item[$\cdot$] Under the condition of partial convexity and domains with increasing cross section,
\item[$\cdot$] For domains ``star-shaped" with respect to an axis.
\end{itemize} 
\end{itemize}
Let us specify some definitions on propagation and blocking properties.
Let $u_\infty(x):=\underset{t\to+\infty}{\lim}u(t,x)$ for all $x\in\Omega$ (see Theorem \ref{existenceThm} in the next section).
\begin{Def}[Blocking and partial propagation]\label{blockpropdef} - Let $u$ be the solution of problem \eqref{problem} satisfying \eqref{condinf}. We say that $u$ is {\em blocked} when 
$$u_\infty(x_1, x')\to0 \text{ as } x_1\to+\infty, \quad (x_1,x')\in\Omega.$$ Otherwise we say that there is {\em partial propagation} in $\Omega$.
\end{Def}
\begin{Def}[Axial partial propagation]\label{axialprop} - Let $u$ be the solution of problem \eqref{problem} satisfying \eqref{condinf}, we say that there is {\em axial partial propagation} if
$$\underset{\R\times B'_R}{\inf} u_\infty>0,$$
for some $R>0$ and where $B'_R$ is the ball of radius $R$ centred at 0 in $\R^{N-1}$.
\end{Def}
\begin{Def}[Complete invasion]\label{completepropdef} - Let $u$ be the solution of problem \eqref{problem} satisfying \eqref{condinf}, we say that there is {\em complete invasion} if $u_\infty\equiv 1$ in $\Omega$.
\end{Def}

This problem is of interest in different domains of application. In population dynamics, to study the evolution of a population going through an isthmus or for fish populations going through a straight for example. In medical sciences to model the evolution of a depolarisation wave in the brain or to study ventricular fibrillations. These motivations from biology and medicine are detailed in section \ref{Motivation}, where we also discuss the interpretations of our results in these contexts.

\subsection{Previous mathematical results}
In this paper, we are interested in the existence and properties of invasion fronts in the direction $x_1$ (see \cite{BH2} for definition) for a bistable reaction-diffusion equation in unbounded domains. This problem was partly studied in \cite{CG} in special cases. There the authors proved  that in two or three dimensions, if $\Omega$ is the succession of two rectangles (or parallelepipeds) with different widths $r$ and $R$, then one can find conditions on $r$ and $R$ such that the solution is blocked when it goes from the small rectangle to the large rectangle. The authors used a specific symmetrization method that does not work for other domains. In \cite{BHM}, the authors proved the existence of generalised transition fronts in exterior domains, which pass an obstacle. They analysed the interactions between travelling wave solutions and different kind of obstacles. In particular, they gave geometrical conditions on the shapes of the obstacles under which there is complete invasion by the most stable state. They also constructed an obstacle that blocks the bistable travelling wave in some area of the domain. In this case, there is only partial invasion.

In this paper, we investigate the question of propagation/blocking properties for rather general ``cylinder-like" geometries. Roughly speaking, by this we mean to say that there is an axis, say the $x_1$--axis, in the direction of which the domain is unbounded (both as $x_1\to-\infty$ and $x_1\to+\infty$). Furthermore, for simplicity, we assume that the domain is a genuine cylinder in the part $\{x_1<0\}$. We first construct a generalised transition front using the techniques of \cite{BHM}. Then the main purpose of this paper is to analyse the geometrical conditions, under which, the wave is blocked or propagates and, in the latter case, whether the invasion is partial or complete.

Related questions have been studied previously in the literature. In \cite{GL} Grindrod and Lewis approximate the normal speed of the interface of a front using the Eikonal approximation and numerical simulations to point out that the presence of a symmetric narrowing and widening will decrease the average speed and that an abrupt widening could lead this speed to 0. More recently in \cite{LMN2006, LMN2013}, Lou, Matano and Nakamura study the effect of undulating boundary on the normal speed of the interface of a front analysing a curvature-driven motion of plane curves in dimension 2. They give necessary and sufficient conditions on the undulations of the domain for the existence of generalised travelling waves in this context.
%Note that to study the existence of invasion fronts we investigate the stability of non constant solution of the stationary problem depending on the geometry of $\Omega$. 
%These questions on the geometric assumptions of $\Omega$ lead to the open problem of the instability of nonconstant steady state solutions of reaction-diffusion equations in unbounded domain, problem that was solved by Matano in \cite{M} and Casten and Holland in \cite{CH} for bounded domains.
%%%%%%%%%%%%%%%%%%%%%%%%%%%%%%%%%%%
\subsection{Notations and assumptions}
Here are some notations that we will use throughout this paper. First $\R^-=(-\infty,0)$ and $\R^+=(0, +\infty)$. The space dimension is $N\geq 2$ and we write the space variable as 
$x=(x_1,x')\in\R^N$ with $ x_1\in\R$ and $x'\in\R^{N-1}$.

We will use balls both in $\R^N$ as in $\R^{N-1}$. Thus to distinguish the two cases, we use the notation
$$B_R(x_0):=\big\{x\in\R^N,\: |x-x_0|<R \big\}$$ 
for the ball of radius $R$ and centred at $ x_0$ in $\R^N$ whereas
$$B_R'(x_0):=\big\{x'\in\R^{N-1},\: |x'-x_0|<R \big\}$$ 
denotes the ball of radius $R$ and centred at $x_0$ in $\R^{N-1}$.
If $x_0=0$, we will simply write $B_R$ and $B_R'$. The outward unit normal at $x\in \partial \Omega$ will be denoted 
$$\nu(x)=(\nu_1(x),\nu'(x))\in\R\times\R^{N-1}.$$ 
Now let us state precisely the conditions on the bistable nonlinearity $f$. First we assume that $f\in C^{1,1}([0,1])$ and there exists $\theta \in (0,1)$ such that 
\begin{subequations}\label{bistable}
\be
f(0)=f(\theta)=f(1)=0,\quad f'(0)<0, f'(1)<0,
\ee
\be
f(s)<0 \quad \text{ for all } s\in(0,\theta),\quad  f(s)>0 \text{ for all } s\in(\theta,1).
\ee
\end{subequations}
Moreover we assume that 1 is ``more stable" than 0 in the sense that:
\be\label{posf}
\int_0^1f(s)ds>0.
\ee
We let $\beta\in(\theta,1)$ denotes the real number:
\be\label{beta}
\beta=\inf\left\{x\in [0,1], \; \int_0^x f(s)ds>0 \right\}.
\ee
As mentioned before, $(\varphi,c)$ will always denote the one-dimension travelling front with $\varphi(0)=\theta$, i.e. $(\varphi,c)$ is solution of \eqref{tw}.

Concerning the domain $\Omega \subset \R^N$, we have in mind to consider cylinders around the $x_1$-axis with varying cross section but the various results are valid for much more general domains and our only general assumptions are that
\begin{gather}
\Omega \text{ is a uniformly } C^{2,\alpha} \text{ domain of } \R^N \text{ and }\label{condOmega1}\\
 \Omega\cap\left\{x\in\R^N, x_1<0\right\}=\R^-\times \omega \text{ where } \omega \text{ is an open subset of } \R^{N-1}. \label{condOmega3}
\end{gather}
Condition \eqref{condOmega3} means that the portion of $\Omega$ lying in a half space (here $\{x_1<0\}$) is a straight cylinder. Thus, the change in the geometry of the domain takes place in the half space $\R^+\times \R^{N-1}$.

%%%%%%%%%%%%%%%%%%%%%%%%%%
\subsection{Main results}
We first prove the well-posedness of the problem with the following theorem. 
\begin{Th}\label{existenceThm}
Let $\Omega\subset\R^N$ and $f$ satisfy conditions \eqref{condOmega1}-\eqref{condOmega3} and \eqref{bistable}-\eqref{posf}, let $(\phi,c)$ be the unique solution of \eqref{tw}, with $c>0$. Then there exists a unique entire solution of 
\begin{equation*}
\begin{cases}
\partial_tu(t,x)-\Delta u(t,x)=f(u(t,x)), &\textrm{for } t\in\R,\quad x\in\Omega,\\
\partial_\nu u(t,x) =0, &\textrm{for } t\in\R,\quad x\in\partial\Omega,
\end{cases}
\end{equation*}
such that
\begin{equation*}\label{uphi}
u(t,x)-\phi(x_1-ct)\to0 \textrm{ as } t\to-\infty \textrm{ uniformly in } \overline{\Omega}.
\end{equation*}
Moreover, $u$ satisfies $u_t(t,x)>0$, $0<u(t,x)<1$ for all $(t,x)\in\R\times\overline{\Omega}$ and $\ds \lim_{t\to +\infty} u(t,.)=u_\infty$ in $C^2_{loc}(\Omega)$ where $u_\infty$ is stationary solution, that is, verifies
\begin{equation}
\label{pbstat}
\begin{cases}
\Delta u_\infty+f(u_\infty)=0  & x\in \Omega,\\
\partial_\nu u_\infty=0  & x\in \partial \Omega.
\end{cases}
\end{equation}
\end{Th}
This theorem is essentially due to \cite{BHM}, who considered an exterior domain. But their proof readily extends to our case. We sketch it in section 3.
\begin{Rk}\label{existencegeo}
This result also applies to bent and right-ended cylinders as those illustrated in section \ref{bentsec}
\end{Rk}
Most of this paper is devoted to analyse the large time behaviour of $u(t,x)$, that is understanding when it is blocked, when there is partial propagation or axial partial propagation or when there is complete invasion. We start with the case of domains with decreasing cross section for which we can prove that there is complete invasion.
\begin{Th}[Complete propagation into decreasing domains]\label{decreaseThm}
Assume that for all $x\in\partial \Omega$, $\nu_1(x)\geq0$, where $\nu_1(x)$ is the first component of the outward unit normal at $x$. Under the same assumptions on $\Omega$, $f$ and $(\phi,c)$ as in Theorem \ref{existenceThm}, the solution $u$ of \eqref{problem} satisfying \eqref{condinf} propagates to 1 in $\Omega$ in the sense of Definition \ref{completepropdef}. That is
$$u(t,\cdot)\to 1 \text{ as } t\to+\infty \text{ in }\Omega.$$
Moreover, if we assume that 
$$\Omega\cap\left\{x\in\R^N, x_1>l\right\}=(l,+\infty)\times \omega_r,$$
for some $l>0$, $\omega_r\subset\R^{N-1}$, then $c$ is the asymptotic speed of propagation, i.e 
\begin{align*}
&\forall\: \hat{c}>c,\hspace{0.1cm} \underset{t\to+\infty}{\lim}\hspace{0.1cm} \underset{x_1>\hat{c}t}{\sup} u(t,x)=0,\\
&\forall\: \hat{c}<c,\hspace{0.1cm} \underset{t\to+\infty}{\lim}\hspace{0.1cm} \underset{x_1<\hat{c}t}{\inf} u(t,x)=1.
\end{align*}
\end{Th}
%Here $B^{N-1}_r(x')$ denotes the ball in $\R^{N-1}$ of center $x'$ and radius $r$. . 
This theorem means that there exists a generalised transition front representing an invasion of 0 by 1 for problem \eqref{problem}.
\begin{Rk}
If we assume that $\Omega$ is a cylinder-like domain with axis on the $x_1$ axis (i.e $w(\cdot)$ defined in \eqref{Omegadef} is open, connected, smooth and bounded), the condition on $\nu_1(x)$ is fulfilled as soon as the cross section is non-increasing with respect to $x_1$, but 
it applies to more general domains (see the upper middle figure of Figure \ref{exomega} for an example in dimension 2). 
\end{Rk}

Next we study the existence of blocking phenomena and prove the following

\begin{Th}[Blocking by a narrow passage] \label{IncreaseThm}
Under the same assumptions on $\Omega$, $f$ and $(\phi,c)$ as in Theorem \ref{existenceThm} and let $a,b$ be fixed with $-\infty<a<b<+\infty$. There exists $\epsilon>0$ small enough depending on 
\be\label{areaindeps}
\Omega\cap\{x\in\R^N,\:b<x_1<b+1\}
\ee 
such that if 
\be\label{areaeps}
|\Omega\cap\{x\in\R^N,\: x_1\in(a,b)\}|<\epsilon,
\ee
then the unique solution $u$ of \eqref{problem} satisfying \eqref{condinf} is blocked in the right part of the domain in the sense of Definition \ref{blockpropdef}. That is
$$u(t,\cdot)\to u_\infty\text{ in }\Omega \text{ as }t\to+\infty\text{ with }u_\infty(x)\to 0\text{ as }x_1\to+\infty.$$
\end{Th}

A particular case of this theorem is the example of abruptly opening cylinders in the direction of propagation $x_1$. More precisely, for a fixed right hand half of the domain $\Omega^+:=\Omega\cap\{x_1>0\}$, there exists $\epsilon>0$ small enough such that if $\Omega\cap\{x_1<0\}\subset \R\times B_\epsilon$, the advancing wave is blocked. From the proof, essentially, we will see that it is blocked at the abrupt passage from a narrow cylinder into a wide domain. On the other hand, Theorem \ref{increaseThmprop} establishes that for a large left cylinder, a generalised travelling front can never be blocked. Thus rescaling the space variable to fix a given measure to the left cylinder and enlarge the right side of the domain does not necessarily work. Figure \ref{widennarrow} displays examples of 2D domains where this blocking phenomenon arises.
\begin{figure}[!h]
\begin{center}
\includegraphics[scale=0.33]{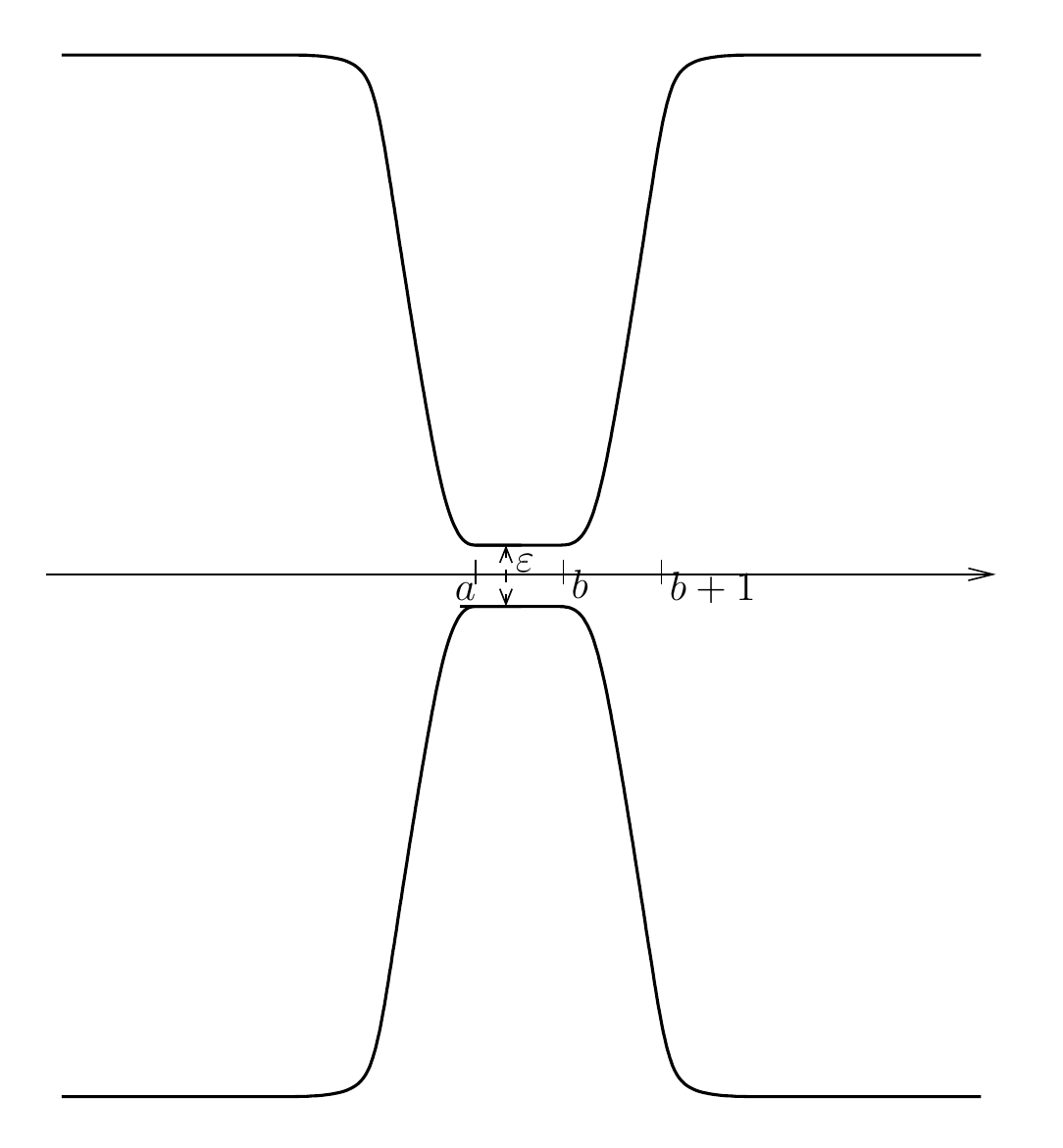}\hspace{1cm}
\includegraphics[scale=0.33]{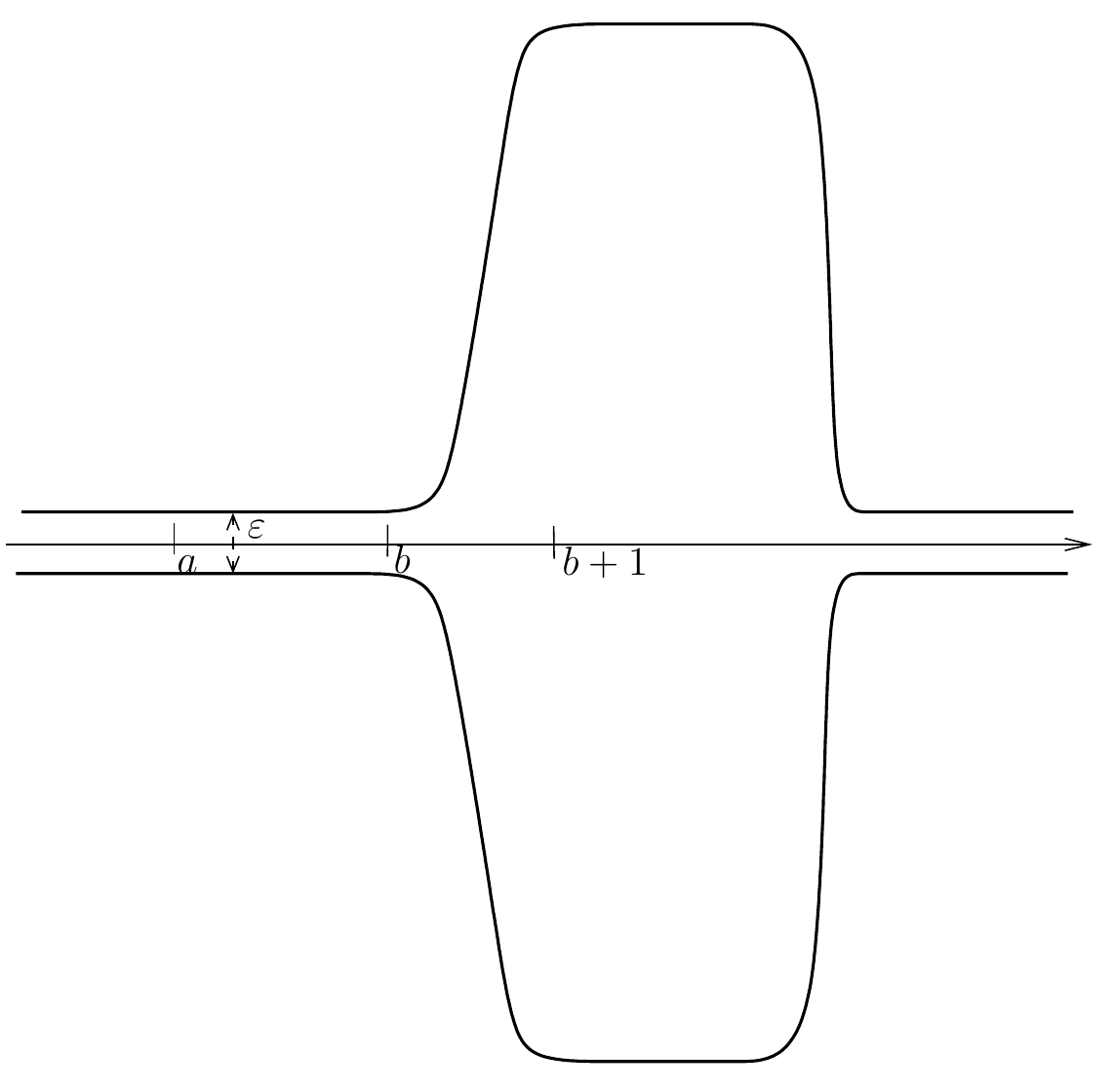}\hspace{1cm}
\includegraphics[scale=0.33]{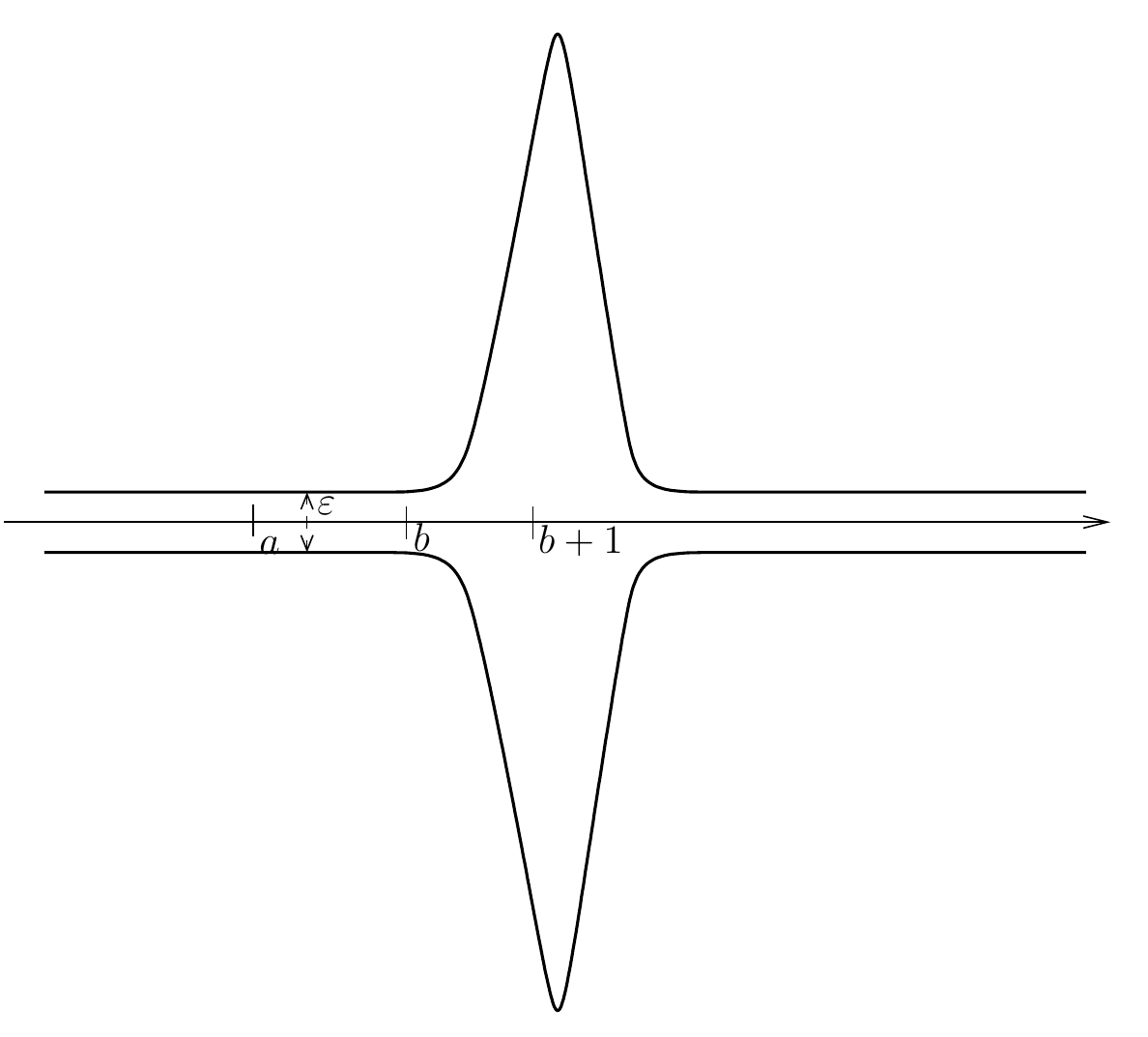}
\caption{Examples of domains in which the wave is blocked at the passage into the wide region, for small enough $\epsilon>0$.}\label{widennarrow}
\end{center}
\end{figure}

\begin{Rk}
Notice that rescaling $f$ into $\lambda f$ for some $\lambda$ positive, then $\epsilon$ becomes $\ds{\frac{\epsilon}{\sqrt{\lambda}}}$, which means that if we increase the amplitude of $f$ then, the threshold $\epsilon$ decreases and blocking will take place in thinner channels.
Furthermore, if we add a diffusion coefficient $D>0$ in front of the second order term (scaled to 1 in this paper) then $\epsilon$ becomes $\sqrt{D}\epsilon$, which means that if $D$ increases then the solution is blocked for larger diameters. 

The blocking phenomenon as well as these scalings can be interpreted in the following manner. As the species moves through a narrow passage, it all of a sudden emerges into a widen open space. As it emerges, diffusion disperses the species, possibly to lower values of density where the reaction becomes adverse. Therefore, too narrow a passage does not allow for the species to reconstitute a strong enough basis to invade the right hand half of the domain and the invasion wave is stuck at this passage. Clearly, a large diffusion accentuates this problem, while a stronger $f$, by allowing large densities to transit through the channel has the opposite effect.
\end{Rk}
Now considering domains that are not narrow, we show that this does not happen if the cross section is large enough. We prove that when $\Omega$ contains a straight cylinder in the $x_1$-direction of radius large enough, there always is propagation.
\begin{Th}[Axial partial propagation]\label{Thmpartial}
Under the same assumptions as in Theorem \ref{existenceThm}, there exists $R_0>0$ such that if $\R\times B'_R \subset \Omega$ for an $R>R_0$, the unique solution $u$ of \eqref{problem} satisfying \eqref{condinf}, propagates along the $x_1$-axis in the sense of Definition \ref{axialprop}. That is 
$$u(t,\cdot)\to u_\infty\text{ in } \Omega \text{ as }t\to+\infty \text{ with  } \inf_{\R\times B'_R} u_\infty>0.$$
\end{Th}
We give more details about how this constant is determined in section \ref{propgeneraldom}.
\begin{Rk}
Notice that in this theorem the invasion may be partial. Indeed depending on the shape of $\Omega$, there could be subdomains where the solution $u_\infty$ is close to 0 (see \cite{BHM, B} and Figure \ref{proppartial} in section \ref{propgeneraldom} of this paper for examples of domains where it happens).
\end{Rk}

In the following theorems, we focus on the geometry $\Omega$ such that the invasion is complete, i.e. $u_\infty\equiv 1$.
The next theorems are proved using a sliding method (either sliding a ball or a cylinder), which is inspired from \cite{BNslide}.
\begin{Th}[Complete propagation in ``star-shaped" domains]\label{Thmslidingcylinder}
We assume that the conditions of Theorem \ref{Thmpartial} are fulfilled. 
Moreover we assume that at each point on the boundary $x=(x_1,x')\in\pO$, the outward unit normal $\nu$ makes a non-negative angle with the direction $x'$. More precisely, writing 
$\nu= (\nu_1, \nu')$ (with $\nu'\in\Rm$), we assume:
\be\label{normal}
\nu' \cdot  x' \geq 0 \qtext{for all points} x=(x_1,x') \in \pO.
\ee 
Then, the invasion is complete, namely, $u_\infty\equiv 1$ in $\O$.
\end{Th}

We call a domain that satisfies the property \pref{normal}  a {\em star-shaped domain with respect to the $x_1$--axis}. One can look at Figure \ref{exomega} (upper left, lower left and middle) for examples of such domains in dimension 2. More examples and counterexamples of such domains are given at the end of section \ref{completeprop}.

The previous Theorem covers a large range of geometries for $\Omega$, however the following Theorem completes this result for another type of geometry. It relies on a different proof.
% and this proof could be easily adapted for any other general domain $\Omega$ satisfying the conditions of Theorem \ref{Thmpartial}. 
\begin{Th}[Complete propagation into increasing domains]\label{increaseThmprop}
Under the same assumptions as in Theorem \ref{existenceThm}, there exists $R_1>0$ such that if\begin{gather*}
\R\times B'_{R}\subset \Omega, \text{ for } R>R_1,\\
\Omega_L^r :=\left\{(x_1,x')\in\Omega,\: x_1>L\right\} \text{ is convex},\\
\{x\in \Omega, \; x_1<L+R\}\subset \R \times B_C,\\
\forall\: x=(x_1,x')\in\partial \Omega \quad  x_1<L+R \Rightarrow \nu_1(x)\leq 0
\end{gather*}
for some $L,C>0$, then the invasion is complete in the sense of Definition \ref{completepropdef}, namely $u_\infty\equiv1$ in $\Omega.$
Here $\nu_1(x)$ is the first component of the outward unit normal at $x$.
\end{Th}
In the proof of this theorem we use a solution of the stationary elliptic equation in a ball with Dirichlet conditions. We then slide this ball in $(\R\times B'_{R})\cup\Omega_L^r$ to get a lower estimate on the solution $u_\infty$ in this domain and then conclude with the maximum principle and Hopf Lemma. This proof could be adapted to somewhat more general domains as soon as it satisfies a sliding ball assumption, i.e every point in the domain could be touch, from within the domain, by a ball of radius $R-\delta$ for some $\delta>0$ such that for $x\in\partial\Omega\cap B_{R}$, $\nu(x)\cdot x\geq0$. 

%% Fin de lecture 28/12/14
We thus point out two different types of behaviour of the solution in the case of widening cylinders. If the diameter of the cylinder before the change in geometry is small and the change is abrupt, we prove that there is blocking, whereas if the cylinder is wide enough before the change of geometry there is propagation even if the domain abruptly increases. One question that remains open is the existence of a sharp threshold, when considering, say a monotone 1-parameter family of domains. The question is thus to know whether the solution is blocked below a critical threshold and propagates above it.

In this paper we consider mainly generalised transition waves. But of course our results and methods bear important consequences for associated Cauchy problems.
\begin{Cor}[The associated Cauchy problem] Let $u$ be the solution of the following Cauchy problem
\begin{equation*}\begin{cases}
\partial_t u(t,x)-\Delta u(t,x)=f(u(t,x)), &\text{for } t\in(0,+\infty),\quad x\in\Omega,\\
\partial_\nu u(t,x)=0,&\text{for } t\in(0,+\infty,) \quad x\in\partial\Omega,\\
u(0,x)=u_0(x),&\text{for } x\in\Omega.
\end{cases}\end{equation*}
Then
\begin{itemize}
\item[$\cdot$] Under the same assumptions as in Theorem \ref{IncreaseThm} on $\Omega$, $f$ and $(\phi,c)$, the propagation of the solution is blocked for any $u_0$, of which the support is contained in the left part of the domain , i.e $supp\{u_0\}\subset\{x\in\Omega,\:x_1<a\}$,
\item[$\cdot$] Under the same assumptions as in either Theorem \ref{decreaseThm}, Theorem \ref{Thmpartial}, Theorem \ref{Thmslidingcylinder} or Theorem \ref{increaseThmprop} on $\Omega$, $f$ and $(\phi,c)$, the propagation results would still hold if the domain $\{x\in\Omega,\: u_0(x)>\theta+\delta\}$ is large enough, for some $\delta>0$.
\end{itemize}
\end{Cor}

This article is organised as follows. In section \ref{Motivation} we detail the motivations from biology and medicine behind this problem. We indicate the interpretations of our findings in these contexts.
We prove the existence of transition fronts (Theorem \ref{existenceThm}) in section \ref{existence section}. We devote section \ref{decreasesection} to the proof of Theorem \ref{decreaseThm} on the propagation of fronts in decreasing domains. 
Section \ref{abruptincreasesection} is devoted to the proof of Theorem \ref{IncreaseThm}, that is the blocking of the front if $\Omega$ contains a narrow passage. Then we derive a sufficient condition under which there is (possibly partial) propagation in section \ref{propgeneraldom}. There, we prove Theorem~\ref{Thmpartial}. Next we turn to further complete propagation results. We prove Theorem~\ref{Thmslidingcylinder} in 
section~\ref{completeprop} and Theorem~\ref{increaseThmprop} in section \ref{completepropwidening}.

%%%%%%%%%%%%%%%%%%%%%%%%%%%%%%%%%%%%%%%%%%%%%%%%%%%%%%%%%%%%%% Relecture du 23/12/14 - reprise 27/12/14 - H

\section{Motivations from medicine and biology and interpretations of the results}\label{Motivation}
%The effect of geometry on propagation of waves arises as a natural question in 
% several contexts of modelling in medicine and biology.
Equation \eqref{problem} is classical in physics, for the representation of phase transitions, and in biology. 
A particularly important motivation for studying the effect of geometry on propagation of waves in equation \eqref{problem} comes from the modelling of \textit{Cortical Spreading Depressions} (CSDs) in the brain. These are transient and large depolarisations of the membrane of neurons that propagate slowly (with a speed of the order of 3mm/min) in the brain. 
CSDs are due to abnormal ionic exchanges between the intra- and extra-cellular space of the neuronal body that slowly diffuse in the brain. There are two stable states, 
the normal polarised rest state and the totally depolarised state, with a threshold on the ionic disturbances for passage from one stable state to the other. After a depolarisation lasting 3 or 4 minutes, several mechanisms take place that repolarise the neurons. 

Mathematical studies of CSDs focus generally on the depolarisation phase. Having in view this phase, CSDs are modelled by a bistable reaction-diffusion equation like \eqref{problem}, 
where $0\leq u\leq 1$ represents the degree of depolarisation. In this context, $u\equiv 0$ stands for the normal polarised state and $u\equiv 1$ represents the completely depolarised state. Thus, a depolarisation is represented by an invasion of the state $u\equiv 1$.
The complete depolarisation mechanism only takes place in the grey matter - the part of the brain that contains the neuronal bodies - while simple ionic disturbances 
%%(and thus the depolarisation)  QUESTION ˆ G.
diffuse everywhere %% Question ˆ G
and are absorbed in the white matter - the part of the brain where no neuron body and only axons are to be found. In this context, the domain $\Omega$ in equation \eqref{problem} can be thought of as representing a portion of grey matter of the brain. The Neumann boundary conditions here is a classical simplification resulting from considering that the grey matter is isolated.

The question of existence of and mechanisms to break  CSDs in human brain is of great importance for the understanding of strokes or migraines with aura among other pathologies. CSDs were first observed by Le\~ao in 1944 \cite{Leao}.
%% QUestion ˆ G : Chez l'homme ?
% When neurons are depolarised, they cannot propagate nerve impulses and this generates several symptoms. For example, CSDs are suspected of being responsible for the aura during migraines with aura. The aura is a set of hallucinations, mainly visual and sensory, which appears during some migraine attacks. In ischemic strokes in rodent, it has been proved that each CSD increases the neurological damage by approximately 30\% \cite{Mies93} and therapies aiming at blocking the appearance of CSDs have shown very promising results in rodent \cite{DeKeyser99, Nedergaard95}. Unfortunately they turned out to be inefficient in humans. 
%Despite intensive research, CSDs have never been clearly observed in the human brain during strokes. 
%% Question ˆ G: Seulement dans le cas des AVC, mais dans d'autres situations %elles ont ŽtŽ observŽes (cf. remarque de Jack Cowan)  
%%% Fin 27/12/14
When neurons are depolarised, they cannot propagate nerve impulses and this generates several symptoms. For example, CSDs are suspected of being responsible for the aura during {\em migraines with aura} \cite{Lauritzen94, Lauritzen01, Hadjikhani01}. The aura is a set of hallucinations, mainly visual and sensory, that appears during some migraine attacks. In ischemic strokes in rodents, it has been proved that each CSD increases the neurological damage by approximately 30\% \cite{Mies93} and therapies aiming at blocking the appearance of CSDs have shown very promising results in rodents \cite{DeKeyser99, Nedergaard95}. Unfortunately they turned out to be inefficient in humans. 

Moreover there are obvious difficulties in observing CSDs in the human brain. These difficulties are related to the complex geometry of the human cortex. With EEG, each electrode records brain electrical activity from a cortical area of several square centimeters whereas the depolarized area during CSDs is much smaller (around few square milimeters). Alternative medical imaging methods record accompagnying symptoms of CSDs (cerebral blood flow increase or decrease) but not the depolarization itself. CSDs can be recorded only in acutely injured human brain undergoing neurosurgery,  using cortical electrode strips similar to the technology applied in epilepsy surgery. But in such a pathological situation, the brain reacts quite differently from a healthy brain.

For many years, it was believed that CSDs were an artifact produced in animal experiments and without significance for human neurological conditions.
Their occurence in human cortex was thus a matter of debate \cite{Aitken91, Gorgi01, Mayevsky96, McLachlan94, Sramka77, Strong02, Back00}. However in the recent years, a medical consensus considering that CSDs occur in human brain has gained ground (see \cite{Lauritzen11} and references therein). Consequently the question of how to block these CSDs became important. For migraines with aura, CSDs have been observed to become extinct when propagating into a sulcus \cite{Bowyer99}, that is in the bottom of a convolution. Generally the grey matter is thicker in a sulcus than in the rest of the cortex. For a CSD during a stroke, understanding how far the CSD will propagate in a specific patient could help in anticipating the dammage caused by the stroke. It could also help in the selection of patients for future clinical trial.

%Thus, their occurence is still a matter of debate \cite{Aitken91, Gorgi01, Mayevsky96, McLachlan94, Sramka77, Strong02, Back00}. 

Mathematical modelling may help in understanding these phenomena since experiments on rodents cannot be conclusive for humans due to the differences of the brain morphologies. 
%Indeed the morphology of the human brain is very different from the rodent brain. 
The grey matter of the human brain composes a thin layer at the periphery of the brain with its thickness subject to abrupt and large variations. The rest of the human brain is composed of white matter. On the opposite, the rodent brain is almost round and entirely composed of grey matter (see Figure \ref{cerveaux}). The variations of the grey matter thickness in the human brain could explain the extinction of CSDs.
%Mathematical modelling may help in understanding these various issues. The difficulties in observing CSDs in the human brain are related to its complex geometry. Only invasive measures could allow us to conclude inambiguously on the existence of CSDs. Except in very special situations (e.g. head surgery), such measures are not available.
%Moreover the morphology of the human brain is very different from the rodent brain. In particular the grey matter of the human brain composes a thin layer at the periphery of the brain with its thicknes subject to large variations. The rest of the human brain is composed of white matter. On the opposite, the rodent brain is almost round and entirely composed of grey matter (see Figure \ref{cerveaux}). The variations of the grey matter thickness in the human brain could explain the inefficiency of therapies aiming at blocking CSDs, 
%since these waves are arlready naturally blocked. 
%by having a blocking effect on the CSD.   
%%%   Question ˆ G
Indeed, as we establish it here, very abrupt variations of the domain may block propagation

\begin{figure}[!h]
\begin{center}
\includegraphics[width=0.3\textwidth]{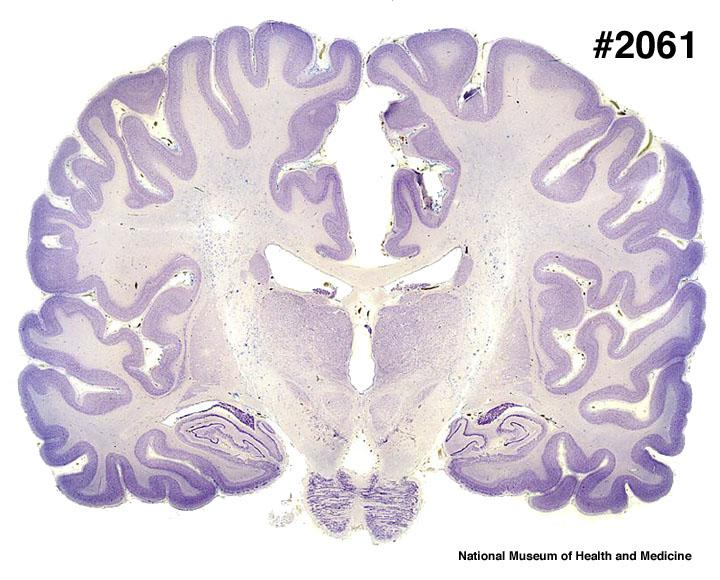}
\hspace{1cm}
\includegraphics[width=0.3\textwidth]{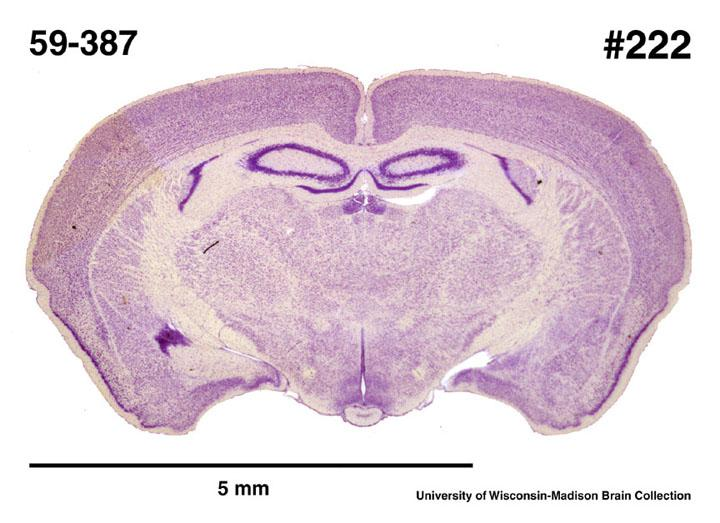}
\caption{Morphology of the brain, cross sections. Left, human brain: grey matter is only a thin layer at the periphery. Right, rodent brain: it is almost entirely composed of grey matter.\\ \small\small{These images are from the University of Wisconsin and Michigan State Comparative Mammalian Brain Collections, and from the National Museum of Health and Medicine, available at the following website http://www.brainmuseum.org/sections/index.html. Preparation of all these images and specimens have been funded by the National Science Foundation, as well as by the National Institutes of Health.} \label{cerveaux}
}

\end{center}
\end{figure}
Chapuisat and Grenier \cite{CG} first stated the hypothesis that the morphology of the human brain may prevent the propagation of CSDs over large distances, unifying several observations by biologists. Several effects may be combined. First, variations of the grey matter thickness may prevent propagation of CSDs. This point is studied numerically in \cite{DDGG} and proved in a special case in \cite{CG}. Furthermore, the absorbing effect of the white matter may stop the propagation of CSDs in areas of the brain where the grey matter is too thin. This phenomenon is studied in \cite{Chap1} and more recently in \cite{BC}. Finally the convolution of the human brain may also enhance the absorption effect of the white matter and thus the blocking of the CSDs. This point is studied numerically in \cite{PMHC09}. Here we prove in a general setting that abrupt transitions from thin to large regions will block waves.

These results can be interpreted in a biological framework. If we assume that the mechanisms triggering CSDs are the same in humans and in rodents, CSDs could be initiated in the human brain but only in areas where the grey matter is large enough and these CSDs will not propagate over large distances since they will be blocked by a sudden enlargement of the grey matter, or if the grey matter becomes too thin as compared to the white matter, or if they go through a sharp convolution of the brain. On the opposite, in rodent brains that are like balls of grey matter, CSDs could be triggered anywhere. 
%This may explain why physicians have not been able to observe CSDs after a stroke in human 
% brains. 
This may explain why it has been difficult to observe CSDs in human brains during a stroke.

In the case of migraines with aura, the symptoms of aura correspond to a dysfunction of certain neurons. When translating the symptoms described by patients on a map of the cerebral areas, it is observed that the dysfunction of neurons propagate at a speed of 3mm/min which makes biologists and physicians think that the aura is due to CSDs. It has been shown that the aura always stops at the bottom of a very deep sulkus.
%Question ˆ G: ne devrait-on pas exoliquer le terme sulkus ? Au fait, qu'est-ce que c'est? 
This phenomenon has been studied numerically in \cite{GDDa} and it strengthens the hypothesis of the role of the morphology of the human brain in the propagation of CSDs. However a complete theoretical study of the influence of the geometry of the grey matter on the propagation of CSDs has never been carried out up to now.

Another motivation for the study of the effect of geomery on propagation in equation~\eqref{problem}
comes from the study of ventricular fibrillations where blocking phenomena in domains with varying thickness is also of great importance. Ventricular fibrillation is a state of electrical anarchy in part of the heart that leads to rapid chaotic contractions of the heart. It is fatal unless a normal rhythm can be restored by defibrillation. One of the medical hypotheses is that ventricular fibrillations are triggered by the generation of a circular excitation wave that are thus trapped in a region of the heart. These waves are not present in a healthy heart and an important medical problem is to understand how these waves can be triggered. Ashman and Hull \cite{AH} suggested that a myocardial infarct could create a circular area of cardiac tissue where excitation waves can propagate {\em only in one way}. Once an excitation wave would enter this area, it would be trapped and the propagation of the wave would trigger ventricular fibrillations. 

Hence an important question is to understand how excitation waves can propagate only in one direction in cardiac fibres. Cranefield \cite{C} asserted that asymmetry of conduction in a normal cardiac fibre occurs due to variations of the diameter of the fibre. This asymmetry is strengthened by depression of excitability. 

Here too, mathematical models of excitation waves in the cardiac tissue may shed light on these phenomena. Cardiac cells as brain cells are excitable and the propagation of the depolarisation initialising the excitation wave can be modelled by a bistable reaction-diffusion equation as in \eqref{problem}. One of the important problems is thus to understand if the geometry of a fibre (or rather a fibre bundle) may trigger an asymmetry in the propagation of the travelling front and if geometrical properties of the domain may even block propagation of travelling waves in one direction but not in the other. This is precisely what we establish here. Indeed, our findings show that in the presence of an abrupt transition, waves can be blocked in one direction (the direction in which the domain is widening) whereas they propagate in the opposite one (the direction in which it is narrowing).  

Related earlier works include numerical studies \cite{Joyner, Miller} that supported this hypothesis. Grindrod and Lewis \cite{GL} studied the propagation of the excitation in a Purkinje fibre bundle using an eikonal equation obtained as an asymptotic  limit of a model of propagation. They used biological values for all parameters and computed the approximated speed of the front. In the context of this asymptotic limit, they proved that both an abrupt increase in the diameter of the bundles as well as an heterogeneous depression of cardiac tissues may prevent propagation of fronts in one direction preserving the propagation in the other direction. 
%However this work is done with an asymptotic model using an eikonal equation and
%it is of importance to understand in which cases a travelling front may or may not 
%propagate along an axis in unbounded domain.

Lastly equation \eqref{problem} is also relevant to describe spatially structured population dynamics. Bistable equations, of the type we study here, model populations that are 
%can also be used to represent the evolution of a species whose dynamics is 
subject to an {\em Allee effect}. One can think of the invasion by a population of fishes,  through a straight into wide ocean.
Invasions of species (plants, trees, animals) subject to Allee effect, going through an isthmus and then into a large area also lead to the kind of geometric effects we study here.

In \cite{BRRK}, the authors studied numerically the evolution of a population in such a complex domain as it faces climate change. 
In this framework, the non-linearity 
in equation \eqref{problem} is space dependent: $f= f(x, u)$ with, say $f(x, u) = - \gamma u$ when $|u| \geq A$ (where $\gamma$ and $A$ are positive constants). Furthermore, to model the effect of climate change one considers $f= f(x-ct, u)$ where $c$ is the exogeneously given climate change velocity. The domain involves a transition region with an opening angle $\alpha$ so that when $\alpha$ is small, the transition is gradual whereas when $\alpha$ is close to $\pi/2$ the transition is abrupt. 
It is shown numerically in \cite{BRRK} that there exists a critical angle $\alpha_0 \in (0, \pi/2)$ such that when  
$\alpha > \alpha_0$ there is extinction of the population, whereas if 
$\alpha < \alpha_0$, the species survives. This property is analogous to what we prove here. Using our methods here we can derive the results when $\alpha$ is either cloe to 0 or to $\pi/2$.  Nevertheless, the complete description of \cite{BRRK}Ê with a specific critical angle is still an open problem. 
%\todo{ + travaux sur population de poisson passant pr un d\'etroit.}
%% Fin lodifications 28/12/14  midi

%%%%%%%%%%%%%%%%%%%%%%%%%%%%%%%%%%%%%%%%%%%%%%%%%%%%%%%%%%%%%%
\section{Existence and uniqueness of the entire solution}\label{existence section}
%%%%%%%%%%%%%%%%%%%%%%%%%%%%%%%%%%%%%%%%%%%%%%%%%%%%%%%%%%%%%%
The proof of Theorem \ref{existenceThm} follows directly from the work of the first author with Fran\c{c}ois Hamel and Hiroshi Matano on waves passing an obstacle \cite{BHM}. We will give here the main points needed to adaptat the arguments but for a detailed proof we refer to \cite[section 2.3-2.4-3]{BHM}. The idea is to construct sub- and supersolutions of problem \eqref{problem} for negative time $t<T<0$ for some $T<0$ and to use these sub- and supersolutions to construct the entire solution. For the uniqueness result one uses the monotonicity in time and a comparison principle.

%%%%%%%%%%%%%%%%%%%%%%%%%%%%%%%%%%%%%%%%%%%%%%%%%%%%%%%%%%
\subsection{Construction of the entire solution}
The construction of the entire solution follows exactly the same steps as in \cite[section 2.3]{BHM}. In our setting, the super- and subsolutions are defined by
\be\label{supersol}
w^+(t,x)=\begin{cases} 2\phi(-ct-\xi(t)), &\textrm{if } x_1>0,\\
						\phi(x_1-ct-\xi(t))+\phi(-x_1-ct-\xi(t)), &\textrm{if } x_1\leq0,
			\end{cases}
\ee
and 
\be\label{subsol}
w^-(t,x)=\begin{cases} 0, &\textrm{if } x_1>0,\\
						\phi(x_1-ct+\xi(t))-\phi(-x_1-ct+\xi(t)), &\textrm{if } x_1\leq0,
			\end{cases}
\ee
where the function $\xi$ is such that $$\dot{\xi}(t)=Me^{\lambda(ct+\xi(t))},\hspace{0.2cm} \xi(-\infty)=0,$$
with $\ds{\lambda=\frac{1}{2}\left(c+\sqrt{c^2-4f'(0)}\right)}$ and $M$ a positive constant to be determined. 
Thus,
$$\xi(t)=\frac{1}{\lambda}\ln\Big(\frac{1}{1-Mc^{-1}e^{\lambda ct}}\Big) \hspace{0.5cm} \text{ for all } t\leq \frac{1}{\lambda c}\ln\Big(\frac{c}{M}\Big).$$
If we assume that $ct+\xi(t)\leq 0$ which is true as soon as $t\leq \frac{1}{\lambda c}\ln\Big(\frac{c}{c+M}\Big)$, it can be proved that there exist $T<0$ and $M>0$ large enough such that $w^+$ is a supersolution and $w^-$ is a subsolution of problem \eqref{problem} for all $t<T$ and $x\in\Omega$.
%We will need the following basics estimates
% \subsubsection*{Basic estimates}
% Let $(\phi,c)$ be the bistable travelling waves solution of problem \eqref{tw}.
% There exist $\alpha_0$, $\alpha_1$, $\beta_0$ and $\beta_1$ positive constants such that 
% \be\label{estphi}
% \begin{cases}
% \alpha_0e^{-\lambda z}\leq \phi(z)\leq\beta_0e^{-\lambda z}, &z>0,\\
% \alpha_1e^{\mu z}\leq 1-\phi(z)\leq\beta_1e^{\mu z}, & z\leq0.
% \end{cases}
% \ee
% And there exist $\gamma_0$, $\gamma_1$, $\delta_0$ and $\delta_1$ positive constants such that
% \be\label{estphider}
% \begin{cases}
% -\gamma_0e^{-\lambda z}\leq \phi'(z)\leq-\delta_0e^{-\lambda z}, &z>0,\\
% -\gamma_1e^{\mu z}\leq \phi'(z)\leq-\delta_1e^{\mu z}, & z\leq0,
% \end{cases}
% \ee
% with $\mu=\frac{1}{2}\Big(-c+\sqrt{c^2-4f'(1)}\Big)$.
% \vspace{0.2cm}\\
% Furthermore as $f\in C^{1,1}$ one has
% \be\label{flip}
% |f(u+v)-f(u)-f(v)|\leq Luv \hspace{0.5cm} \forall\hspace{0.1cm} 0\leq u,v\leq1,
% \ee
% where $L>0$ some constant.
% \bigskip\\
% Then we use these basic estimates to prove that 
% $$\partial_t w^+(t,x)-\Delta w^+(t,x)-f(w-+(t,x))\geq0,\quad\forall t<T,\:x\in\Omega,$$
% $$\partial_t w^-(t,x)-\Delta w^-(t,x)-f(w^-(t,x))\leq0,\quad\forall t<T,\:x\in\Omega,$$
% using different arguments for $x_1>0$, $ct+\xi(t)\leq x_1\leq0$, $x_1<ct+\xi(t)$ (see \cite[section 2.3]{BHM} for details).

Now the idea is to construct a sequence of solutions $u_n$ of the Cauchy problem defined for $-n\leq t<+\infty$, such that $u_n$ converges toward an entire solution as $n\to+\infty$. Following \cite[section 2.4]{BHM}, we define $u_n$ to be the solution of \eqref{problem} for all $t\geq-n$ with the initial condition
$$u_n(-n,x)=w^-(-n,x).$$
Since $w^-\leq w^+$ and using the comparison principle one proves that $w^-\leq u_n\leq w^+$ and 
$$u_n(t,x)\geq u_{n-1}(t,x) \text{ for all }t\in[-n+1,+\infty), x\in\Omega.$$
Using the monotonicity of the sequence and parabolic estimates as $n\to+\infty$, we see that $u_n$ converges to an entire solution $u$ of \eqref{problem} for all $t\in\R$, $x\in\Omega$, and 
$$w^-(t,x)\leq u(t,x)\leq w^+(t,x) \text{ for all } t\in(-\infty,T], x\in\Omega.$$
Since $\xi(-\infty)=0$, we have thus proved that there exists an entire solution to \eqref{problem} such that 
\begin{equation*}
|u(t,x)-\phi(x_1-ct)|\to0 \text{ as } t\to-\infty \text{ uniformly in } x\in\Omega.
\end{equation*}

%%%%%%%%%%%%%%%%%%%%%%%%%%%%%%%%%%%%%%%%%%%%%%%%%%%%%%%%%%
\subsection{Time monotonicity}

As in \cite{BHM}, we will make use of the property that $u_t>0$ for all $t\in\R$ and $x\in\Omega$. This can be proved  using the maximum principle and the fact that $w^-_t>0$. Hence
$$(u_n)_t(t,x)>0 \text{ for all } t\in(-n,+\infty), x\in\Omega.$$
Letting $n\to+\infty$ and using the strong maximum principle, one concludes that 
$$u_t(t,x)>0 \text{ for all } t\in\R, x\in\Omega.$$
Since $0<u<1$, using monotonicity and usual parabolic estimates, we obtain that $u(t,.)\to u_\infty$ in $C^2_{loc}(\Omega)$ and $u_\infty$ is a solution of 
\eqref{pbstat}.

In the following, we will need more informations on $u_t$ so
for $0<\eta\leq\frac{1}{2}$, we introduce
$$\Omega_\eta(t)=\{x\in\Omega:\hspace{0.1cm} \eta\leq u(t,x)\leq1-\eta\}.$$ \begin{Lemma}\label{uniqsol}
For all $\eta\in(0,\frac{1}{2}]$, there exists $\delta>0$ such that 
\be\label{utpos}
u_t(t,x)\geq \delta \hspace{0.3cm} \text{for all } t\in(-\infty,T_\eta] \text{ and } x\in\Omega_\eta(t).
\ee
\end{Lemma}
\begin{proof}[Proof of Lemma \ref{uniqsol}]
The proof of the lemma is the same as in \cite{BHM} section 3 except that the domain is changed. We will therefore give only the main ideas. 

First using \eqref{condinf}, there exist $T_\eta\in\R$ and  $M_\eta>0$ such that,
\be\label{omegaeta}
\Omega_\eta(t)\subset \{x\in\Omega, \; |x_1-ct|<M_\eta\}\subset \R^-\times \omega \textrm{ for all } -\infty<t\leq T_\eta.
\ee
Now argue by contradiction and assume that there exist $t_k\in(-\infty,T_\eta]$ and $x_k\in\Omega_\eta(t)$ such that $u_t(t_k,x_k)\to0$ as $k\to+\infty$. There are two possibilities:
\begin{itemize}
\item If $(t_k)_k$ is bounded, then $x_k$ is also bounded and up to extraction of a subsequence, we can assume that $t_k\to t_*$ and $x_k\to x_*$ so $u_t(t_*,x_*)=0$. But $u_t >0$ on $\Omega$ so $x_*\in \partial \Omega$ and using the parabolic Hopf lemma, $\partial_\nu u_t=(\partial_\nu u)_t<0$ and we obtain a contradiction with the boundary condition.
\item If $(t_k)_k$ is unbounded, then $x_k$ is also unbounded which means here that $x_{1,k}$ tends to $-\infty$. Defining $u_k(t,x)=u(t+t_k,x_1+x_{1,k},x')$ and using parabolic estimates, we get that $u_k\to u^*$ as $k\to+\infty$ in $C^{1,2}_{\text{loc}}(\R\times \R \times \omega)$. Then, we conclude as in the previous case. 
\end{itemize}
\end{proof}
%%%%%%%%%%%%%%%%%%%%%%%%%%%%%%%%%%%%%%%%%%%%%%%%%%%%%%%%%%
\subsection{General comparison principle and uniqueness of the solution}

We now want to prove that the entire solution satisfying condition \eqref{condinf} is unique. In a sense, this result is an extension of the parabolic comparison principle to the case when the initial condition is given at $t=-\infty$, which we now state in the following more general form.

\begin{Lemma}[General comparison principle] \label{princomp}
Assume $u$ and $v$ are super- and subsolutions of \eqref{problem} with limiting condition \eqref{condinf}, that is:
\be\label{supersubproblem}
\begin{cases}
\partial_t u-\Delta u \geq f(u) \quad  \text{in } \R\times\Omega,\\
\partial_\nu u\geq0  \quad\text{on }\R\times\partial\Omega,\\
\displaystyle \lim_{t\to -\infty} \inf_{x\in \Omega} u(t,x)-\phi(x-ct) \geq 0 .
\end{cases} %\ee
\! \text{and }
%\be\label{subproblem}
\begin{cases}
\partial_t v-\Delta v \leq f(v) \quad \text{in } \R\times\Omega,\\
\partial_\nu v\leq0 \quad \text{on }\R\times\partial\Omega,\\
\displaystyle \lim_{t\to -\infty} \sup_{x\in \Omega} v(t,x)-\phi(x-ct) \leq 0 
\end{cases}
\ee
Then $v\leq u$ for all $t\in \R$ and $x\in \overline{\Omega}$.
\end{Lemma}
This comparison principle indeed contains the uniqueness of the entire solution satisfying \eqref{condinf}.
\begin{proof}[Proof of Lemma \ref{princomp}]
In the following proof, we will note $\L w:=w_t-\Delta w-f(w)$.

Choose $\eta>0$ small enough so that 
\begin{equation}
\label{derf}
f'(s)\leq -\gamma \text{ for } s \in [-2\eta, 2\eta]\cup [1-2\eta,1+2\eta]
\end{equation}
for some $\gamma>0$. For any $\epsilon\in (0,\eta)$ we can find $t_\epsilon \in \R$ such that 
$$
v(t,x)-u(t,x)\leq \epsilon \text{ for } -\infty<t\leq t_\epsilon \text{ and } x\in \Omega.
$$
We let $T^*=\min(T_\eta-\sigma\epsilon,t_\epsilon)$.
For any $t_0\in (-\infty,T^*]$, we define
$$z(t,x)=u(t+t_0+\sigma\epsilon(1-e^{-\gamma t}),x)+\epsilon e^{-\gamma t}$$
where the constant $\sigma>0$ will be specified later. Then 
\begin{equation}
\label{CIsub}
v(t_0,x)\leq z(0,x) \text{ for } x\in \Omega.
\end{equation}
Let us prove that $z$ is a supersolution in the time range $t\in [0,T^* -t_0]$ if $\sigma$ is sufficiently large (independently of $\epsilon$).
\begin{eqnarray*}
\L z&\geq&  \sigma  \gamma\epsilon e^{-\gamma t}u_t-\gamma\epsilon e^{-\gamma t}+f(u)-f(u+\epsilon e^{-\gamma t}) \\
&=& \epsilon e^{-\gamma t}(\sigma  \gamma u_t-\gamma-f'(u+\theta\epsilon e^{-\gamma t}) ),
\end{eqnarray*}
for some $\theta=\theta(t,x)$ with $0 < \theta <1$.

If $x\in \Omega_\eta [t+t_0+\sigma\epsilon(1-e^{-\gamma t})]$, then using the preceding lemma \ref{uniqsol}, we see that
$$
\L z\geq \epsilon e^{-\gamma t}(\sigma  \gamma \delta -\gamma-\max_{0\leq s\leq 1}f'(s)).
$$
Therefore, $\L z>0$ if $\sigma$ is chosen sufficiently large (independently of $\epsilon>0$). 
If $x\not\in \Omega_\eta [t+t_0+\sigma\epsilon(1-e^{-\gamma t})]$, we have 
$u(t+t_0+\sigma\epsilon(1-e^{-\gamma t}), x ) +\theta \epsilon e^{-\gamma t}\in [0,2\eta]\cup[1-\eta,1+\eta]$. Consequently, $f'(u+\theta \epsilon e^{-\gamma t})\leq \gamma$ and 
$$
\L z\geq \epsilon e^{-\gamma t}(-\gamma +\gamma)=0.
$$
Applying the classical comparison principle for $t\in [0,T^*-t_0]$ and $x\in \Omega$, we get
$$
v(t_0+t,x)\leq z(t,x) \text{ for } t\in [0,T^*-t_0], x\in \Omega.
$$
Rewriting $t_0+t$ as $t$ for simplicity of notation, we see that
$$
v(t,x)\leq u(t+\sigma\epsilon(1-e^{-\gamma (t-t_0)}),x)+\epsilon e^{-\gamma (t-t_0)} \text{ for } t\in [t_0,T^*], x\in \Omega
$$
where $t_0$ is chosen arbitrarily in $(-\infty, T^*]$. Letting $t_0 \rightarrow -\infty$, we derive
$$
v(t,x)\leq u(t+\sigma \epsilon,x)
$$
for all $t\in (-\infty,T^*], x\in \Omega$ and we can once again apply a classical comparison principle and conclude that this inequality holds for all $t\in \R$. Finally, letting $\epsilon \to 0$, we have proved $v\leq u$. 
\end{proof}

\subsection{More general domains}\label{bentsec}
The proofs we have sketched here actually work for more general domains.  
For instance we can consider {\em bent} cylindrical type domains in which we have existence and uniqueness of an entire solution of problem \eqref{problem} satisfying \eqref{condinf}.  By slightly modifying the sub- and supersolutions $w^-$ and $w^+$ for $x_1<0$, we can extend these results to U-shaped cylindrical domains. Of course in these type of domains, condition \eqref{condinf} is satisfied for only one of the end of the cylinder near $x_1\to-\infty$. We also have the same existence and uniqueness proof for domains that only have one infinite ``end'' near $x_1\to -\infty$ and the remaining part is bounded. These different types of domains are illustrated in Figure \ref{bend}. %%%
%  Fin relecture 29/12/14  ˆ  20h30
%%%
\begin{figure}[!h]
\begin{center}
\includegraphics[scale=0.63]{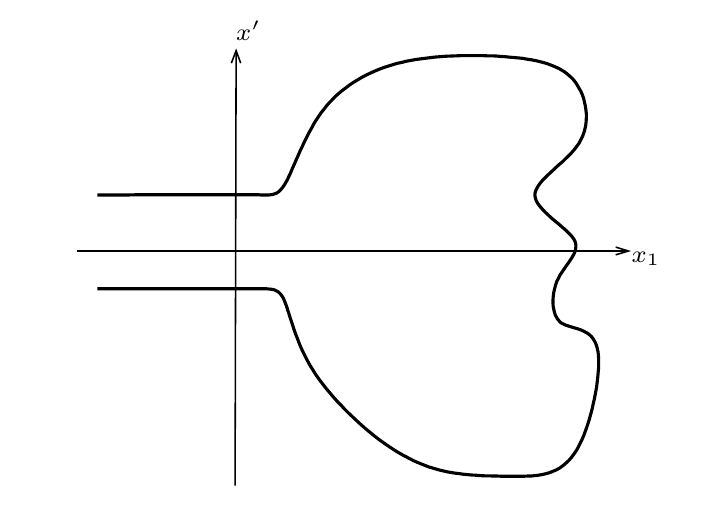}
\includegraphics[scale=0.62]{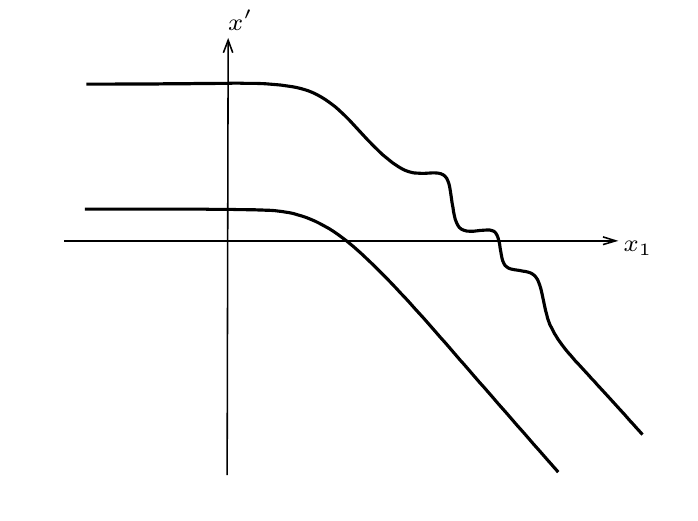}
\includegraphics[scale=0.63]{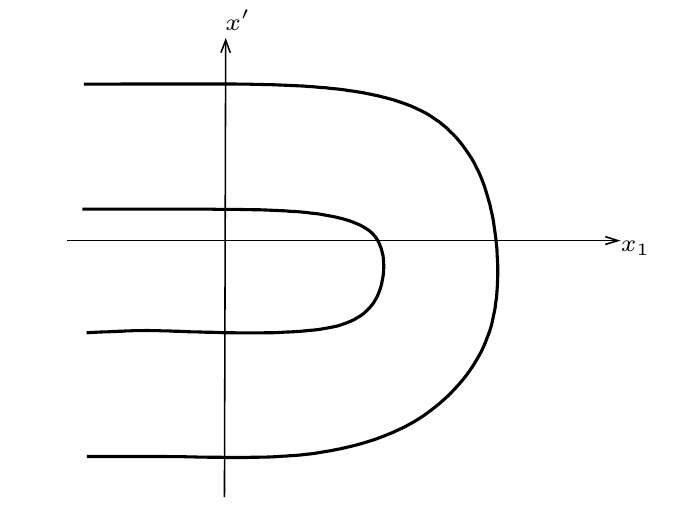}
\caption{Examples of 2D domains for which existence and uniqueness of an entire solution of Problem \eqref{problem}, satisfying \eqref{condinf}, hold. {\em Left}: domain with one end ;  {\em center}: bent cylindrical type domain ; {\em right}: U-shaped cylindrical domain.}\label{bend}
\end{center}
\end{figure}

%%%
%  Fin relecture 02/01/15  ˆ  18h30
%%%

%%%%%%%%%%%%%%%%%%%%%%%%%%%%%%%%%%%%%%%%%%%%%%%%%%%
%%%%%%%%%%%%%%%%%%%%%%%%%%%%%%%%%%%%%%%%%%%%%%%%%%%

% %%%%%%%%%%%%%%%%%%%%%%%%%%%%%%%%%%%%%%%%%%%%%%%%%%%%%%%%%%%%%%
% \section{Propagation and asymptotic speed}\label{speedsection}
% %%%%%%%%%%%%%%%%%%%%%%%%%%%%%%%%%%%%%%%%%%%%%%%%%%%%%%%%%%%%%%
% This section is devoted to the proofs of theorems \ref{decreaseThm} and \ref{perturbationThm} which present which present cases where the geometry of the domain does not affect really the propagation of the travelling front. In particular, the spreading speed is conserved.  

%%%%%%%%%%%%%%%%%%%%%%%%%%%%%%%%%%%%%%%%%%%%%%%%%%%%%%%
\section{Propagation in domains with decreasing cross section} \label{decreasesection}
In this section we consider propagation in the direction of a narrowing domain. We formulate this condition by requiring the first component of the outward normal to be non-negative, see figure \ref{decrease} for examples in 2D. Here we prove Theorem \ref{decreaseThm}. 
\begin{figure}[!h]
\begin{center}
\includegraphics[scale=0.5]{OmegaRetrbis}
\includegraphics[scale=0.4]{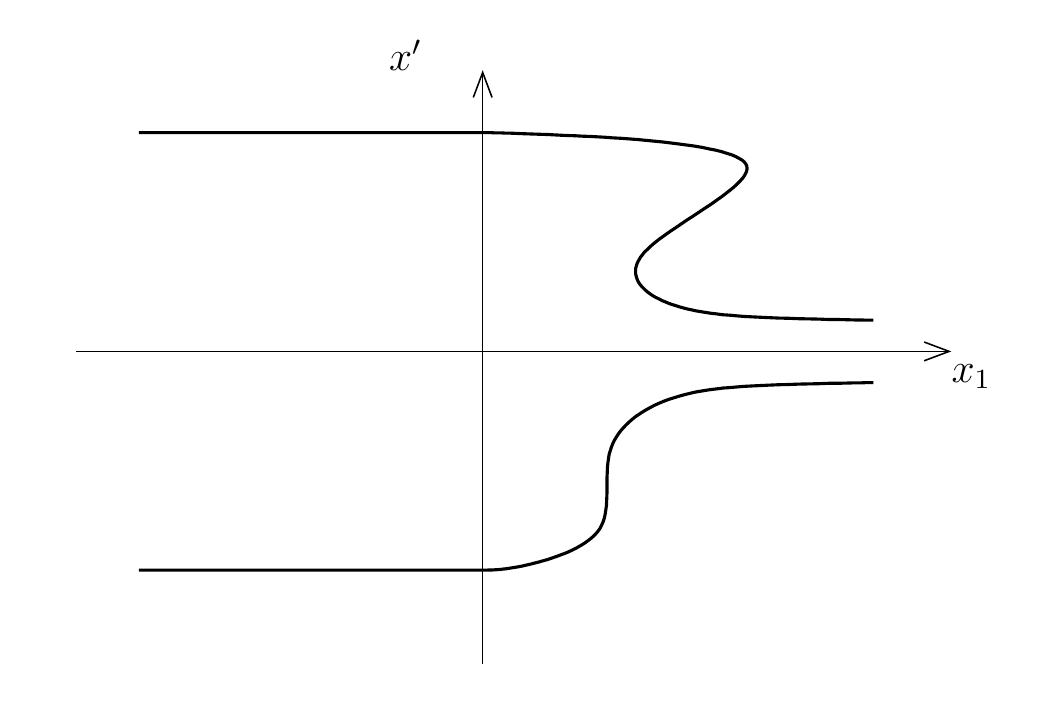}
\includegraphics[scale=0.4]{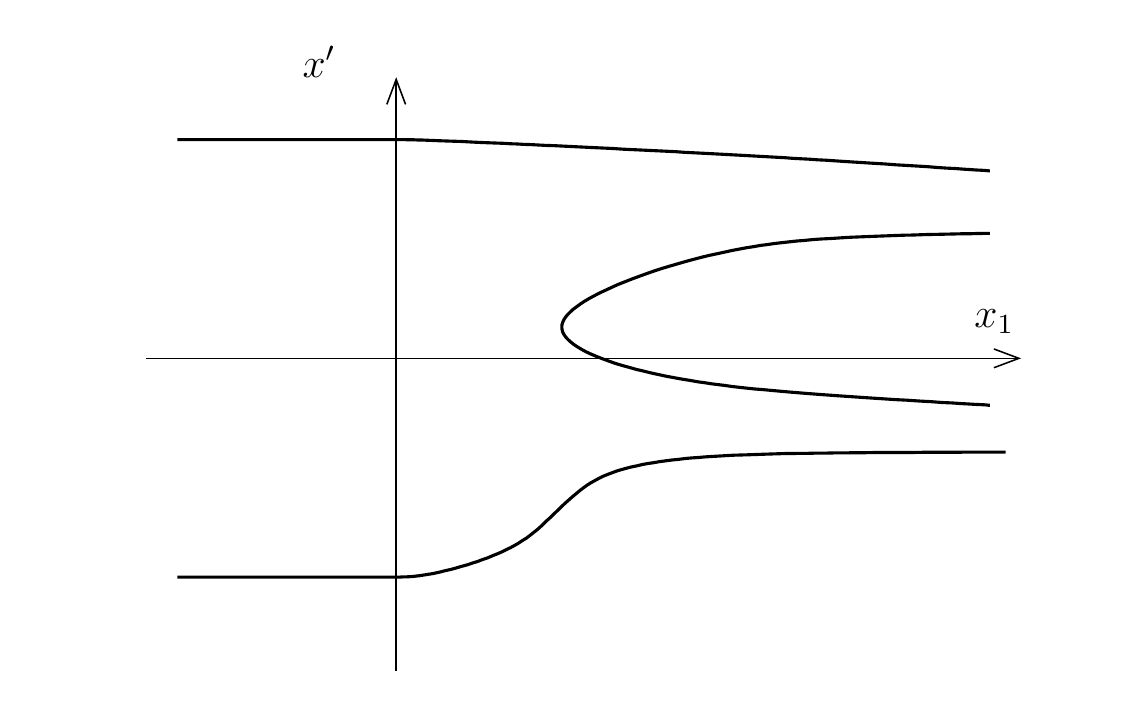}
\caption{Examples of 2D domains concerned by theorem \ref{decreaseThm}.}\label{decrease}
\end{center}
\end{figure}

We start by proving that the unique solution $u$ of (\ref{problem}) such that \eqref{condinf} is satisfied cannot be blocked and propagates to 1 at large times. The proof is based on the following lemma.
\begin{Lemma}\label{philequ}
Under the same assumption as in Theorem \ref{decreaseThm}, in particular that the domain $\Omega$ satisfies $\nu_1(x)\geq0$ for all $x\in\partial\Omega$, let $(\phi,c)$ be the travelling wave solution defined by \eqref{tw} and $u$ be the unique solution of \eqref{problem} such that \eqref{condinf} is satisfied. Then for all $t\in\R$ and $x\in\Omega$
$$\phi(x_1-ct)\leq u(t,x).$$
\end{Lemma}
\begin{proof}[Proof of Lemma \eqref{philequ}] 
Define $v(t,x)=\phi(x_1-ct)$ for all $(t,x)\in\R\times\Omega$, $v$ satisfies the following problem:
\be\label{subsolv}
\begin{cases}
\partial_t v-\Delta v=f(v) & \text{in } \R\times\Omega,\\
\partial_\nu v\leq0 &\text{on }\R\times\partial\Omega,\\
v(t,x)-\phi(x_1-ct)\to0 & \textrm{as } t\to-\infty \textrm{ uniformly in } x\in\overline{\Omega}.
\end{cases}
\ee
Indeed, for all $t\in\R$, $x\in\partial\Omega$,
$$\partial_\nu v(t,x)=\phi'(x_1-ct)\nu_1(x).$$
As assumed in Theorem \ref{decreaseThm}, $\nu_1(x)\geq 0$ for all $x\in\partial\Omega$ and since $\phi$ is decreasing, we see that $\partial_\nu v\leq0$ in $\R\times\partial\Omega$. 
Applying the general comparison principle of Lemma \ref{princomp}, we get $v\leq u$.
\end{proof}
Next, we  make precise the behaviour of the solution $u$ at large time, i.e when $t\to+\infty$. Since $u$ is bounded between 0 and 1 and $\phi(x_1-ct)\to 1$ when $t\to +\infty$ uniformly locally for $x\in \Omega$, we have $u(t,x)\to 1$ uniformly locally for $x\in \Omega$. 

To complete the proof of Theorem \ref{decreaseThm}, we now furthermore assume that $\Omega\cap\{x\in \R^N, \; x_1>l\}=(l,+\infty)\times \omega_r$ for a given $l>0$, $\omega_r\subset\R^{N-1}$. We will prove that the 1D speed $c$ is the asymptotic spreading speed i.e. 
\begin{align}
&\text{for all } \hat{c}>c \quad \lim_{t\to+\infty} \sup_{x_1>\hat{c}t} u(t,x)=0,\label{limitec0}\\
&\text{for all } \hat{c}<c \quad \lim_{t\to+\infty} \inf_{x_1<\hat{c}t} u(t,x)=1.\label{limitec1}
\end{align}
The  lower bound for the speed of propagation in \eqref{limitec1} immediately results from $u(t,x)\geq \phi(x_1-ct)$ for all $(t,x)\in\R\times\Omega$.

We now turn to the proof of \eqref{limitec0}, which gives an upper bound for the speed of propagation. Let us fix $\hat{c}>c$. For all $\eta\in (0, \frac \theta 2)$, introduce $f_\eta$ (see Figure \ref{fetafig}) a smooth function in $[\eta,1+\eta]$ such that ${f(\eta)=f(\theta)=f(1+\eta)=0}$, $f<0$ on $(\eta,\theta)$, $f>0$ on $(\theta,1+\eta)$ and
\be\label{feta}
\begin{cases}
f_\eta\equiv f &\text{in } [2\eta,1-\eta],\\
f_\eta>f &\text{in } [\eta,2\eta)\cup(1,1+\eta].
\end{cases}
\ee
%\end{column}
%\begin{columns}
\begin{figure}[!h]
\begin{center}
\includegraphics[scale=0.7]{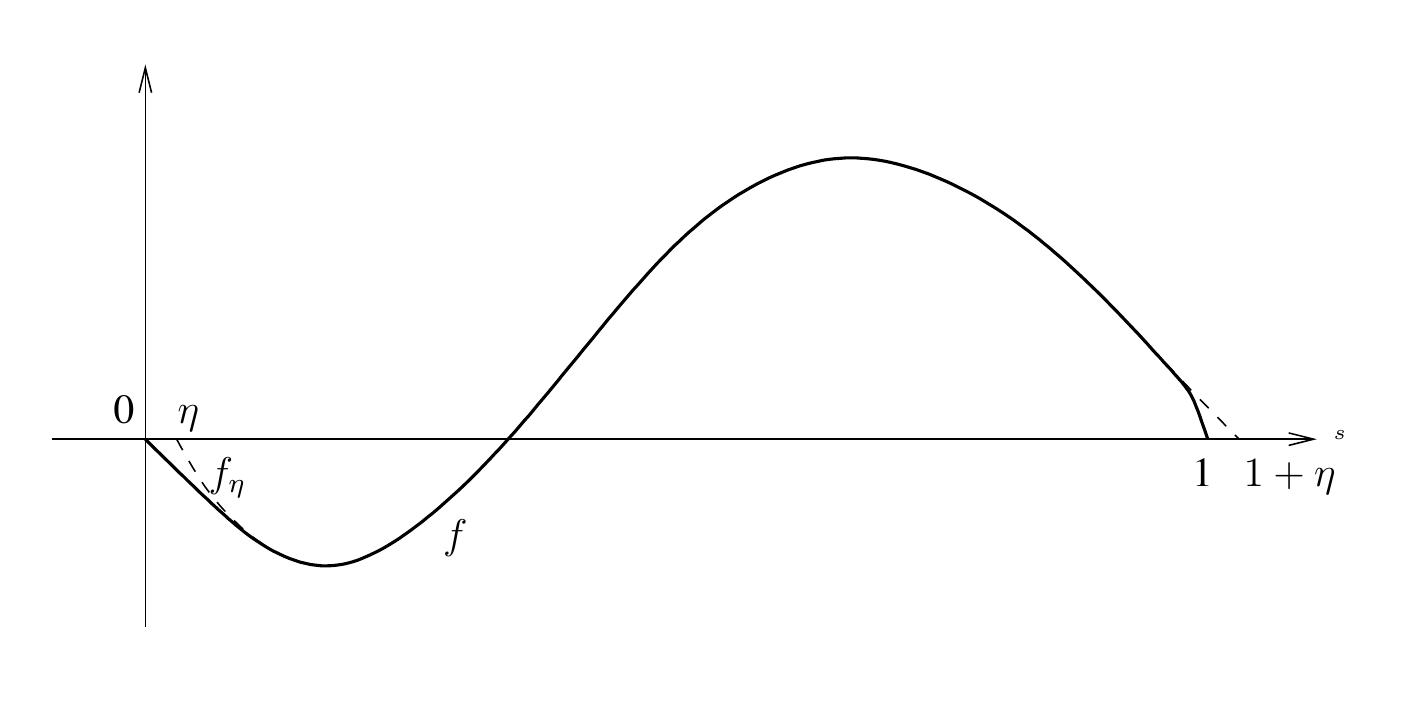}
\caption{Example of function $f_\eta$ that satisfies \eqref{feta}.}\label{fetafig}
\end{center}
\end{figure}
%\end{column}
%\end{multicols}
Consider the 1-D travelling front solution $\phi_\eta$ of speed $c_\eta$ associated with $f_\eta$, that is $(\phi_\eta,c_\eta)$ satisfies the equation
\be
\begin{cases}
\phi_\eta''(z)+c_\eta\phi_\eta'(z)+f_\eta(\phi_\eta)=0 &\text{for all } z\in\R,\\
\phi_\eta(-\infty)=1+\eta, \hspace{0.2cm} \phi_\eta(+\infty)=\eta.
\end{cases}
\ee
It is well know that $c_\eta\to c$ as $\eta\to0$. 
We denote $\eta_0>0$ the threshold such that $c_\eta<\hat{c}$ for all $\eta\in (0,\eta_0)$. Let $\eta\in (0,\eta_0)$ be fixed.
Since $|u(t,x)-\phi(x_1-ct)|\to 0$ as $t\to-\infty$ and $\lim_{+\infty} \phi=0$,  there exist $T<0$ and $L>0$ such that for all $t<T$ and $x_1>L$, $u(t,x)\leq \eta$.

% So $u$ satisfies the following problem
% \be\label{paraboliccylinder}
% \begin{cases}
% u_t(t,x)-\Delta u(t,x)=f(u(t,x)), &\text{for } t\in[T,+\infty[, x\in [l,+\infty[\times \omega_r,\\
% \partial_\nu u(t,x) =0, &\text{for } t\in[T,+\infty[, x\in [l,+\infty[\times \partial \omega_r,\\
% \end{cases}
% \ee
% where $0<l<L$ such that $\Omega\cap\{x,x_1>l\}=[l;+\infty)\times \omega_r$ as assumed in the Theorem \ref{decreaseThm}, with $\omega_r$ a compact set of $R^{N-1}$, diam($\omega_r$)=$r$ and with
% \be\label{Dirichlet}
% u(t,l,x')\leq 1,\text{ for } t\in[T,+\infty[.
% \ee
% Define $g:[l,+\infty[\to[\epsilon,1]$ a smooth function that is equal to 1 in $[l,L]$ and equal to $\epsilon$ in $[L+1,+\infty[$, then $u$ satisfies
% \be\label{condinitDirichlet}
% u(T,x)\leq g(x_1), \text{ for } x \in[l,+\infty[\times \omega_r.
% \ee
% \vspace{0.1cm}\\
Let us consider the following problem
\be\label{paraboliccylinderv}
\begin{cases}
v_t(t,x)-\Delta v(t,x)=f(v(t,x)), &\text{for } t\in[T,+\infty[, x\in [l,+\infty[\times \omega_r,\\
\partial_\nu v(t,x) =0, &\text{for } t\in[T,+\infty[, x\in [l,+\infty[\times \partial \omega_r,\\
v(t,l,x')=u(t,l,x'), &\text{for } t\in[T,+\infty[, x'\in\omega_r\\
v(T,x)= u(T,x), &\text{for } x \in[l,+\infty[\times \omega_r.
\end{cases}
\ee
It is constructed in such a way that $u$ is a solution of it and $v=\phi_\eta(x_1-c_\eta t-\xi)$ is a supersolution if $\xi\in \R$ is chosen large enough so that $\phi_\eta(l-\xi)\geq 1$.
According to the comparison principle, $v\geq u$.
%\todo{Verifier principe de comparaison sur domaine a angle. OK ici au voisinage de l'angle car $\phi()>1$ la.}
 Thus
$$
\lim_{t\to+\infty} \sup_{x_1>\hat{c}t} u(t,x)\leq \lim_{t\to+\infty} \sup_{x_1>\hat{c}t} v(t,x)=\eta
$$
since $\hat{c}>c_\eta$. As $\eta$ can be chosen arbitrarily small in $(0,\eta_0)$, we infer from this that 
$$
\lim_{t\to+\infty} \sup_{x_1>\hat{c}t} u(t,x)=0
$$
and \eqref{limitec0} is proved.

\begin{Rk}
The same kind of result can be obtained for smooth perturbations of straight cylinders. Precisely, using techniques from \cite{B}, one can easily prove that if $\Omega_\epsilon$ is a family of domains that converges in the $C^{2,\alpha}$ topology to a straight cylinder as $\epsilon \to 0$, then $u_\epsilon $ the solution of \eqref{problem} with condition \eqref{condinf} converges to  1 everywhere. Moreover if $\nu_1(x)\leq 0$ for all $x\in\partial\Omega$ and $\Omega_\epsilon\cap \{x_1>l\}=(l,+\infty)\times \omega_{\epsilon,r}$, the solution $u_\epsilon$ propagates 
%with a constant speed $c_\epsilon$ that tends to
with the 1-D speed $c$.
% as $\epsilon \to 0$.
\end{Rk}
\section{Blocking by a narrow passage} \label{abruptincreasesection}
%%%%%%%%%%%%%%%%%%%%%%%%%%%%%%%%%%%%%%%%%%%%%%%%%%%%%%%%%%%%%%
In this section we prove Theorem \ref{IncreaseThm}.  

There are earlier results on the existence of non constant stable steady states for bistable reaction-diffusion equations with Neumann boundary conditions when there is a narrow passage in a number of cases. 
Matano \cite{M} first showed the existence of non-constant stable solutions in the case of a bounded domain having the shape of an hourglass. The first author of the present paper  together with Hamel and Matano \cite{BHM} established an analogous result for an exterior domain having a narrow passage to a confined region. The paper \cite{BHM}  exhibits a non constant stable steady state in an exterior domain with narrow passage where the area in which the solution is close to 0 is bounded whereas the area where the solution is close to 1 is unbounded. 

Our proof is inspired from \cite{BHM}. 
 Here, we construct a non constant stable steady state for domains where the two areas are unbounded (the one where the solution is close to 0 and the one where the solution is close to 1).
Stability is obtained since the solution is constructed by minimization of an energy.

Actually, we derive a more precise result here. We will show that once the measure of $\Omega\cap\{b<x_1<b+1\}$ is fixed, there exists $\epsilon>0$ small enough such that if $|\Omega\cap\{a<x_1<b\}|<\epsilon$ then there exists a supersolution $w$ of the elliptic problem 
on  $\Omega\cap\{x_1 >a \}$ such that $w=1$ on $\Omega\cap\{x_1=a\}$
and $w$ goes to 0 in $\Omega$, as $x_1\to+\infty$. 

Thus, we consider the reduced problem in $\ds{\Omega'=\{x\in \Omega, \; x_1>a\}}$, (see 
Figure \ref{Omega'}).
\begin{figure}[!h]
\begin{center}
\includegraphics[scale=0.5]{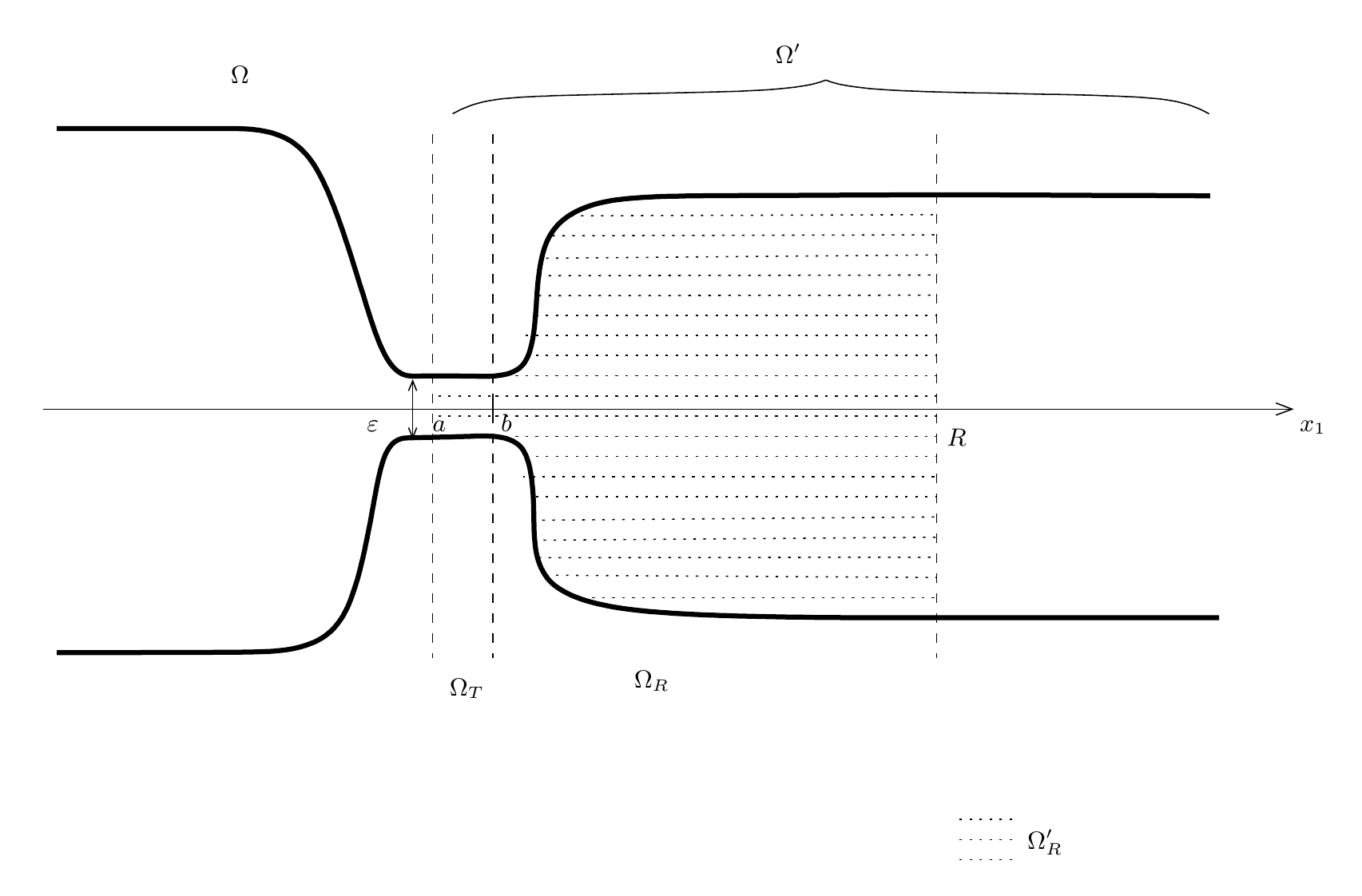}
\includegraphics[scale=0.5]{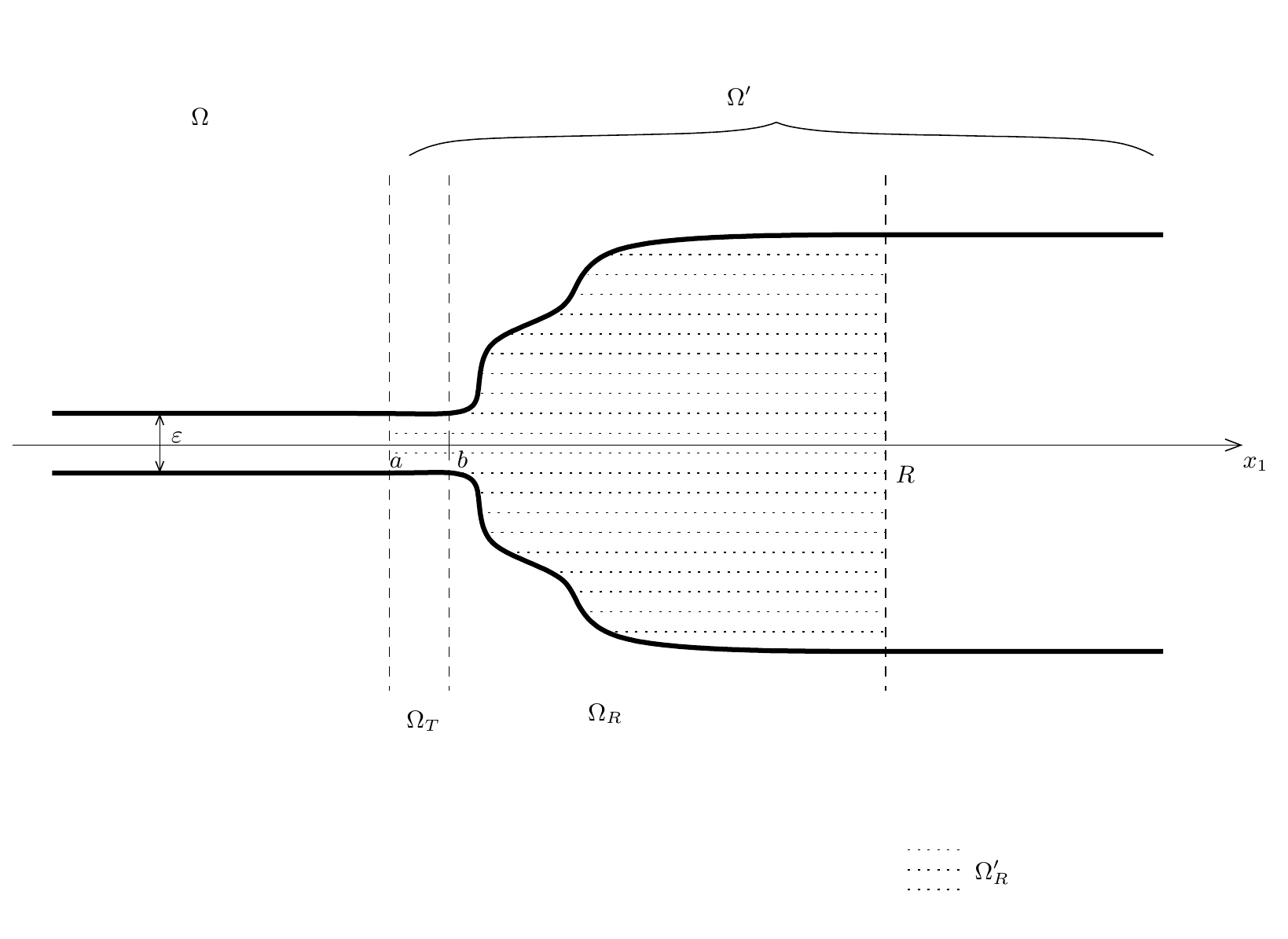}
\caption{\footnotesize Illustrations of the various domains used in the proof of Theorem \ref{IncreaseThm} for two different $\Omega$ satisfying assumptions \eqref{areaeps} and \eqref{areaindeps} in Theorem \ref{IncreaseThm}. $\Omega$ is the initial domain, $\Omega'=\{x\in\Omega, x_1>a\}$, $\Omega'_R=\{x\in\Omega, a<x_1<R\}$, $\Omega_T=\{x\in \Omega, a<x_1<b\}$ and $\Omega_R=\{x\in \Omega, b<x_1<R\}$}\label{Omega'}
\end{center}
\end{figure}
\begin{equation}\label{eqOmega'}
\begin{cases}
\Delta w+f(w)=0 & \textrm{ in } \Omega',\\
\partial_\nu w=0 &\textrm{ on } \partial \Omega'\backslash\{x_1=a\},\\
w\equiv1 &\textrm{ on } \{x_1=a\},\\
\end{cases}
\end{equation}
Such a $w$ blocks the propagation of 
the solution $u$ of \eqref{problem} satisfying \eqref{condinf}. Indeed, by the comparison principle, $u(t,x) \leq w(x)$ in $\Omega'$, whence $u_\infty(x_1,x')\to 0$ as $x_1\to+\infty$.

To find a solution of \eqref{eqOmega'} we proceed in two steps. We first restrict our analysis to a bounded domain $\Omega'_R=\Omega'\cap\{x_1<R\}$, $R\geq b+1$ (see figure \ref{Omega'}) and prove in section \ref{ReducedPb} that there exists a function $0<w_R<1$ in $\Omega'_R$ solution of:
\begin{equation}\label{reducprob}
\begin{cases}
\Delta w+f(w)=0 & \textrm{ in } \Omega'_R,\\
\partial_\nu w=0 &\textrm{ on } \Gamma=\partial \Omega'_R\backslash\big( \{x_1=a\}\cup\{x_1=R\}\big),\\
w\equiv1 &\textrm{ on } \{x_1=a\},\\
w\equiv0 &\textrm{ on } \{x_1=R\}.\\
\end{cases}
\end{equation}
 Then we prove in section \ref{Rinfinity} that $w_R\to w_\infty$ as $R\to+\infty$ with  
%where $w_\infty$ is a solution of \eqref{eqOmega'}
%\begin{equation*}
%\begin{cases}
%\Delta w_\infty+f(w_\infty)=0 & \textrm{ in } \Omega'=\Omega\cap\{x_1>-1\},\\
%\partial_\nu w_\infty =0 &\textrm{ on } \Gamma'=\partial \Omega'\backslash %{x_1=-1\},\\
%w_\infty\equiv1 &\textrm{ on } \{x_1=-1\},
%\end{cases}
%\end{equation*}
%such that 
$w_\infty\to0$ as $x_1\to+\infty$ and conclude, using the general comparison principle, that the solution $u$ of \eqref{problem}-\eqref{condinf} is blocked by $w_\infty$ in the right hand part of the domain.

%%%%%%%%%%%%%%%%%%%%%%%%%%%%%%%%%%

\subsection{Reduced problem}\label{ReducedPb}
For $D\subset \R^N$, we introduce the energy functional defined for all $w\in H^1(D)$,  by
\be\label{energy}
J_D(w)=\int_D \frac{1}{2}|\nabla w|^2+F(w)dx,
\ee
where $F(t)=\int_t^1f(s)ds$. Note that using \eqref{posf} and extending $f$ linearly outside $[0,1]$, $F(t)>0$ for all $t\in\R$ and that $F$ has the shape illustrated in Figure \ref{F}.			
\begin{figure}[!h]
\begin{center}
\includegraphics[scale=1]{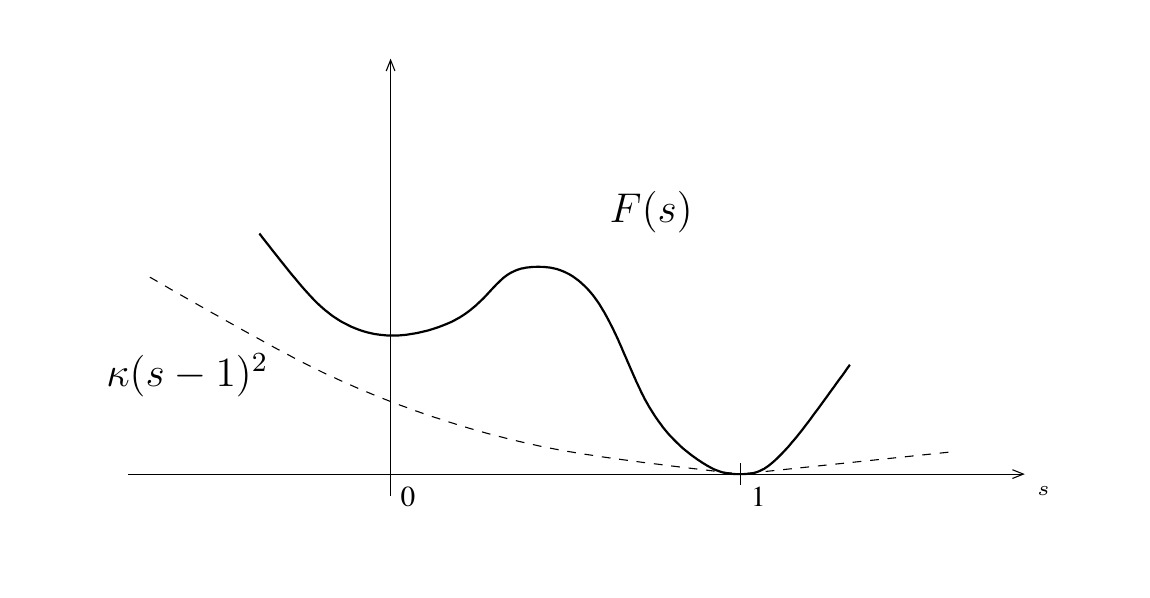}
\caption{\footnotesize{The function $F(s)=\int_s^1f(t)dt$, where $f$ is a bistable function for $s\in[0,1]$ and is extended linearly outside $[0,1]$.}}\label{F}
\end{center}
\end{figure}

Let $\Gamma_L=\{x\in \partial\Omega'_R, x_1=a\}$ and $\Gamma_R=\{x\in\partial\Omega'_R, x_1=R\}$, we define 
$$H^1_{1,0}=\Big\{ w\in H^1(\Omega'_R),\hspace{0.1cm} w\equiv 1 \text{ on } \Gamma_L \text{ and } w\equiv 0 \text{ on } \Gamma_R\Big\}.$$
Finding a local minimizer of $J_{\Omega'_R}$ in an open subset of $H^1_{1,0}$ yields a stable solution of \eqref{reducprob}.

We start by proving that the constant function $w_0\equiv 0$ 
on $\Omega_R=\{x\in\Omega, \; b<x_1<R\}$ is a local minimum of $J_{\Omega_R}$ in $H^1(\Omega_R)$.
%\vspace{0,2cm}\\
\begin{Prop}\label{omegaR}
There exist $\delta>0$ and $\alpha>0$ such that for all $R\geq b+1$, for all $w\in H^1(\Omega_R)$ with ${\lVert w\rVert_{H^1(\Omega_R)}\leq \delta}$ one has 
\be\label{minenergie}
J_{\Omega_R}(w)\geq J_{\Omega_R}(w_0)+\alpha\lVert w \rVert^2_{H^1(\Omega_R)}.
\ee
\end{Prop}
\begin{proof}[Proof of Proposition \ref{omegaR}]
%\vspace{0.2cm}\\
%\begin{gather*}
We first observe that,
\be\label{EnergieR0}
J_{\Omega_R}(w)=\int_{\Omega_R}  \frac{1}{2}|\nabla w|^2+F(0)+F'(0)w+\frac{F''(0)}{2}w^2+\eta(w)w^2,
\ee
where $\eta:\R\to\R$ is a continuous function such that $\eta(s)\to 0$ as $s\to 0$. This can also be written as
\be\label{EnergieR1}
J_{\Omega_R}(w)=J_{\Omega_R}(w_0)+\int_{\Omega_R}  \frac{1}{2}|\nabla w|^2-\frac{f'(0)}{2}w^2+\eta(w)w^2.
\ee
%\end{gather*}
To exploit \eqref{EnergieR1}, we require the following lemma:
\begin{Lemma}\label{eta}
If $\eta:\R\to\R$ is the function defined in \eqref{EnergieR0}, we have,
\be
\left|\int_{\Omega_R}\eta(z)z^2\right| \leq \mu(z)\lVert z\rVert^2_{H^1(\Omega_R)},
\ee
where $\mu(z)\to0$ as $\lVert z\rVert_{H^1(\Omega_R)}\to 0$
\end{Lemma}
\begin{proof}[Proof of Lemma \ref{eta}]
By extending $f$ linearly at infinity, we may assume that $F$ is quadratic at infinity, which implies that $\eta$ is a bounded function from $\R$ to $\R$, i.e there exists $C\in\R$ such that $|\eta(s)|\leq C$ for all $s\in\R$. Let us fix $p>0$ sufficiently small so that $H^1(\Omega_R)\subset L^{p+2}(\Omega_R)$ independently of $R\geq b+1$.  
Since $ \eta$ is bounded and $\eta(s)\to0$ as $s\to0$, for all $\gamma>0$ there exists $\xi>0$ such that 
\be\label{eta1}
|\eta(t)|\leq \gamma +\xi t^p.
\ee
Hence, 
\be\label{eta2}
\left|\int_{\Omega_R}\eta(z)z^2\right| \leq\gamma\lVert z\rVert^2_{H^1(\Omega_R)} +\xi \lVert z\rVert^{p+2}_{L^{p+2}(\Omega_R)}.
\ee
By using Sobolev embedding we infer that
\be\label{eta3}
\left|\int_{\Omega_R}\eta(z)z^2\right| \leq (\gamma+\xi \lVert z\rVert^{p}_{H^1(\Omega_R)})\lVert z\rVert^2_{H^1(\Omega_R)}.
\ee
Since $\gamma$ can be chosen as small as we want, this inequality yields Lemma~\ref{eta}. 
\end{proof}
As $f'(0)<0$, we can fix $ \alpha=\frac 12 \min(\frac{1}{2},-\frac{f'(0)}{2}) >0$. Thus,
$$\int_{\Omega_R}  \frac{1}{2}|\nabla w|^2-\frac{f'(0)}{2}w^2\geq 2\alpha\lVert w\rVert^2_{H^1(\Omega_R)}$$
and from the previous lemma, there exists $\delta>0$ small enough such that for all $\lVert w\rVert^2_{H^1(\Omega_R)}\leq \delta$, one has $\eta(w)<\alpha$. Substituting in \eqref{EnergieR1} yields 
\be\label{EnergieR2}
J_{\Omega_R}(w)\geq J_{\Omega_R}(w_0)+\alpha\lVert w\rVert^2_{H^1(\Omega_R)}.
\ee
and proposition \ref{omegaR} is proved. 
\end{proof}
Now going back to the domain $\Omega'_R$, we extend the function $w_0$ by
\be\label{omega0}
w_0(x)=w_0(x_1):=\begin{cases} \frac{|x_1-b|}{b-a} &\textrm{for }  x_1\in[a,b),\quad x\in\Omega,\\
0 &\textrm{for } x_1\in[b,R], \quad x\in\Omega.
\end{cases}
\ee
The function $w_0$ is in $H^1_{1,0}$. The next step is to prove the following proposition:
\begin{Prop}\label{minenergtout}
There exists $\delta>0$, such that for all $w\in H^1_{1,0}$ with $\lVert w-w_0\rVert_{H^1(\Omega'_R)}=\delta$ one has 
\be\label{Energie4}
J_{\Omega'_R}(w)> J_{\Omega'_R}(w_0).
\ee
\end{Prop}
%This Proposition implies that $J_{\Omega'_R}$ admits a local minimiser in the open set $$\ds{\{ w\in H^1_{1,0}, \lVert w-w_0\rVert_{H^1(\Omega'_R)}<\delta\}},$$ and this minimiser $w_R$ is a stable solution of \eqref{reducprob}. Moreover $w_R\not\equiv0$ and $w_r\not\equiv1$ if $\delta$ is small enough, thus by the strong maximum principle $0<w_R<1$.
\begin{proof}[Proof of Proposition \ref{minenergtout}] 
In the following, all constants will be denoted generically by $C$. 
For all $w\in H^1_{1,0}$,
$$J_{\Omega'_R}(w)- J_{\Omega'_R}(w_0)=\underbrace{J_{\Omega_T}(w)- J_{\Omega_T}(w_0)}_{(1)}+\underbrace{J_{\Omega_R}(w)-J_{\Omega_R}(w_0)}_{(2)}.$$
If $\lVert w-w_0\rVert_{H^1(\Omega'_R)}=\delta$, one has that $\lVert w-w_0\rVert_{H^1(\Omega_R)}=\lVert w\rVert_{H^1(\Omega_R)}\leq\delta$ and one can use Proposition \ref{omegaR} to prove that the second term $(2)$ above is bounded from below, i.e
$$J_{\Omega_R}(w)-J_{\Omega_R}(w_0)\geq\alpha \lVert w-w_0\rVert^2_{H^1(\Omega_R)}.$$
For term (1), as $f(0)=0$, $f'(0)<0$ and $F(1)=0$, there exists $\kappa>0$ such that

$$F(s)\geq\kappa (s-1)^2,$$
for all $s\in\R$ (see Figure \ref{F}). This implies that 
\be\label{Energie5}
J_D(w)\geq\nu\lVert w-1\rVert^2_{H^1(D)},
\ee
for all $D$ compact domain in $\Omega'_R$, $w \in H^1(D)$, where $\nu=\min\{\kappa,\frac{1}{2}\}$. Moreover 
\be\label{Energie6}
J_{\Omega_T}(w_0)=\int_{\Omega_T}\frac{1}{2}|\nabla w_0|^2+F(w_0)\leq C|\Omega_T| \text{ and } |\Omega_T|\leq \epsilon.
\ee
From \eqref{Energie5} and \eqref{Energie6}, we obtain
$$J_{\Omega_T}(w)- J_{\Omega_T}(w_0)\geq \nu\lVert w-1\rVert^2_{H^1(\Omega_T)}-C\epsilon.$$
Then we also use
\begin{align*}
\nu\lVert w-1\rVert^2_{H^1(\Omega_T)}&\geq \frac{\nu}{2}\lVert w-w_0\rVert^2_{H^1(\Omega_T)}-\nu\lVert w_0-1\rVert^2_{H^1(\Omega_T)}\\
&\phantom{==}\geq\frac{\nu}{2}\lVert w-w_0\rVert^2_{H^1(\Omega_T)}-C\epsilon.
\end{align*}
Therefore there exists $\gamma>0$ such that 
$$J_{\Omega'_R}(w)- J_{\Omega'_R}(w_0)\geq \gamma \lVert w-w_0\rVert^2_{H^1(\Omega'_R)}-C\epsilon.$$
Taking $\epsilon>0$ small enough, we get
$$J_{\Omega'_R}(w)- J_{\Omega'_R}(w_0)>0.$$
This completes the proof of Proposition \ref{minenergtout}
\end{proof}
From Proposition \ref{minenergtout}, we get the existence of a local minimizers $w_R$ of the energy functional $J_{\Omega'_R}$ that belongs to $H_{0,1}^1$ such that $\lVert w_R-w_0\rVert<\delta$. This yields a stable solution of \eqref{reducprob}. For $\delta>0$ small enough, we have $0<w_R<1$ using the maximum principle.
\begin{Rk} Note that in this proof we could actually choose any fonction $g(x')$ on $\Omega\cap\{x_1=a\}$ and make the same construction and prove that there exists a solution of
\begin{equation*}\begin{cases}
\Delta w+f(w)=0, &\text{in }\Omega'_R,\\
\partial_\nu w=0 &\textrm{on } \Gamma=\partial \Omega'_R\backslash\big( \{x_1=a\}\cup\{x_1=R\}\big),\\
w\equiv g(x') &\textrm{on } \{x_1=a\},\\
w\equiv0 &\textrm{on } \{x_1=R\}.\\
\end{cases}
\end{equation*}
for $\epsilon$ small enough.
\end{Rk}
The next step will be to study the behaviour as $R\to+\infty$. Let us recall that $\delta$ is fixed for all $R\geq b+1$.

%%%%%%%%%%%%%%%%%%%%%%%%%%%%%%%%%%%%%%%%%%%%%%%%%%%%%
\subsection{Construction of a particular supersolution}\label{Rinfinity}
In this section we construct a particular supersolution of \eqref{problem}. We start with the following Proposition:
\begin{Prop}\label{uinftyprop}
Let $w_R$ be the minimizer of the energy functional $J_{\Omega'_R}$ defined in the previous section. When $R\to+\infty$, up to a subsequence, $w_R$ converges to a solution $w_\infty$ of
\be\label{superprob}
\begin{cases}
\Delta w_\infty+f(w_\infty)=0 & \textrm{ in } \Omega',\\
\partial_\nu w_\infty =0 &\textrm{ on } \Gamma=\partial \Omega'\backslash \{x_1=a\},\\
w_\infty\equiv1 &\textrm{ on } \{x_1=a\},
\end{cases}
\ee
such that $w_\infty\to0$ as $x_1\to+\infty$.
\end{Prop}
\begin{proof}[Proof of Proposition \ref{uinftyprop}:]
As $0<w_R<1$ and using Schauder estimates, there exists $R_n\to+\infty$ as $n\to+\infty$ such that $w_{R_n}\to w_\infty$ in $C^2_{loc}$ as $n\to+\infty$ and $w_\infty$ is a solution of (\ref{superprob}). It remains to prove that the limit $w_\infty$ satisfies $w_\infty\to0$ as $x_1\to+\infty$. Using Fatou's Lemma:
\begin{align*}
\lVert w_0-w_\infty\rVert^2_{H^1(\Omega')}\leq \liminf_{n\to+\infty}\lVert (w_0-w_{R_n}) \mathds{1}_{\{ -1<x_1<R_n \} }\rVert^2_{H^1(\Omega')}\leq \delta^2
\end{align*}
Then arguing by contradiction, we assume that there exist $\eta>0$ and a sequence $(x_n=(x_{n,1},x'_n))_{n\in\N}$ such that $\ds{x_n\in\overline{\Omega}}$ for all $n\in\N$, $x_{n,1}\to+\infty$ as $n\to+\infty$ and $w_\infty(x_n)>\eta$, for all $n\in\N$. As $w_\infty\in C^2_{loc}$, it implies that $|\nabla w_\infty|\leq C$ on every compact set $K$. For all $x\in B(x_n,\frac{\eta}{2C})\cap\Omega$, 
\begin{align*}
|w_\infty(x)-w_\infty(x_n)|\leq \underset{x\in B(x_n,\frac{\eta}{2C})\cap\Omega}{\max}\lVert\nabla w_\infty\rVert |x-x_n|
\end{align*}
So $w_\infty(x)\geq \frac{\eta}{2}$, for all $x\in B(x_n,\frac{\eta}{2C})\cap\Omega$, for all $n\in\N$. This yields 
$$\lVert w_\infty-w_0\rVert_{L^2(K)}\geq \frac{\eta}{2}|B_{\frac{\eta}{2C}}\cap\Omega|\times (\textrm{ number of } x_n\in K).$$
As $\Omega$ is assumed to be uniformly $C^1$,  $\lVert w_\infty-w_0\rVert_{L^2(K)}\geq \delta$ for $K$ large enough, which is impossible.
\end{proof}
We now extend $w_\infty$ by 1 outside $\Omega'$, i.e 
\be\label{uinfty}
\tilde{w}_\infty(x)=\begin{cases} w_\infty(x) &\textrm{if } x\in\Omega',\\
						1 &\textrm{if } x\in\Omega\backslash\Omega'.\\
\end{cases}
\ee
Then $\tilde{w}_\infty$ is a supersolution of the parabolic problem
\be\label{para}
\begin{cases}
v_t-\Delta v=f(v) &\textrm{in } \Omega,\\
\partial_\nu v=0 &\textrm{on } \partial\Omega,
\end{cases}
\ee
Moreover $\ds \lim_{t\to -\infty}\inf_{x\in \Omega} \tilde{w}_\infty -\phi(x_1-ct)\geq 0$. Indeed for $x_1\leq a$, $\tilde{w}_\infty -\phi(x_1-ct)=1-\phi(x_1-ct)\geq 0$ for all $t\in \R$ and for $x_1\geq a$, $\ds \lim_{t\to -\infty}\phi(x_1-ct)=0$ uniformly in $x_1\geq a$. 
Using the generalised maximum principle \ref{princomp}, we obtain $u(t,x)\leq \tilde{w}_\infty(x)$ for all $t\in\R$ and $x\in\Omega$. As $u$ is increasing in $t$ and bounded between 0 and $\tilde{w}_\infty(x)$, $u(t,x)\to u_\infty(x)$ as $t\to+\infty$ for all $x\in\Omega$ and passing to the limit in time, $u_\infty(x)\leq \tilde{w}_\infty(x)$ for all $x\in\Omega$. We have thus proved Theorem \ref{IncreaseThm}.

%%%%%%%%%%%%%%%%%%%%%%%%%%%%%%%%
%%%%%%%%%%%%%%%%%%%

\section{Partial propagation in domains containing a sufficiently large cylinder} \label{propgeneraldom}
%%%%%%%%%%%%%%%%%%%%%%%%%%%%%%%%%%%%%%%%%%%%%%%%%%%
%In this section we are interested in the existence of propagation phenomena in infinite cylinder without assuming any monotonicity on the diameter of the cylinder.

In this section we investigate the propagation properties of $u$ when $\Omega$ contains an infinite cylinder in the $x_1$-direction with large enough radius.  We will prove Theorem~\ref{Thmpartial}. We assume here that $\R\times B'_{R}\subset\Omega $ where $R>R_0$ for some $R_0$ which we now define. 
 %defined in appendix \ref{prerequisite}.

We consider the following semi-linear elliptic equation with Dirichlet condition in the ball $B_R$ of radius $R$ in $\R^N$ with $N\geq 2$:
 \be\label{eqphi}
\bc
-\Delta z = f(z)  &\text{in } B_R,\\
 z = 0 &\text{on } \partial B_R.\\
\ec \ee
%In this paper, we make a large use of the ``ground state'' solutions. In this section we collect the %properties of these functions that we use in here. 
%ze refer to \cite{BL} for details and proofs of these properties.
For large $R$, positive solutions can be found by minimization of the energy functional:
$$
J(z):= \int_{B_{R}} \frac{|\nabla z(x) |^2}{2}-\left(\int_0^{z(x)}f(s)ds\right)dx
$$
in ${H}^1_0(B_R)$. It is known \cite{BL} that there exists a threshold $R_0>0$ such that for $R<R_0$, there is no positive solution of \eqref{eqphi} while for all $R\geq R_0$, there exists at least one positive solution of \eqref{eqphi}. Moreover, for all $R\geq R_0$, there exists a maximum positive solution which we denote by $z=z_R$ in the sequel.  
 Incidentally, from a result of Rabinowitz \cite{rabinowitz}, it can be inferred that  for all $R>R_0$, Problem \eqref{eqphi} has at least two distinct solutions. %$W_R$. 
 
The maximum solution satisfies $0<z_R<1$ in $B_R$ and is radially symmetric and decreasing with respect to $|x|$ \cite{GNN}. We extend $z_R$ by 0 outside $B_R$ and denote $z_R(x)= w_R(|x|)$ where $w_R$ is defined on $\R^+$. In particular, we know that $\sup z_R= z_R(0) = w_R(0) >\theta$. 
%Moreover it is shown in \cite{BL} that this maximal solution converges to 1 locally uniformly as $R%\to+\infty$. 

Using the function $w_R$ we will now derive the following estimate from below on $u_\infty$ for all
$R\geq R_0$:
\begin{Lemma}\label{comppropcyllemma}
For all $x=(x_1,x')\in \Omega$, we have
\be\label{borneinf}
u_\infty(x)\geq w_R(|x'|).
\ee
\end{Lemma}
%u_\infty (x_1, x') \geq w_R (|x'|) \qtext{for all} x_1\in \R, \, x'\in\R^{N-1}.
%\ee
\begin{proof}
This is obtained by a simple sliding argument. We know that for fixed $R$, we have 
$\max w_R < 1$. Therefore, for large enough $a>0$, we have $ u_\infty (x_1-a, x')Ê> \max w_R \geq w_R(|x|)$ for
all $x= (x_1, x')\in B_R$, hence everywhere in $\Omega$. Let $u^a (x_1, x') := u_\infty (x_1-a, x')$. Both functions $u^a$ and $z_R$ are solutions of the same elliptic equation in $B_R$ and $z_R=0$ while $u^a >0$  on the boundary of $B_R$. We also know that $u^a > z_R$ for large enough $a$. We now start decreasing $a$. From the strong maximum principle we infer that  the function $z_R$ has to stay below $u^a$ for all $a\in \R$. Consequently,
\[
u_\infty (x_1, x') \geq \max_{a\in \R} z_R(x_1 +a, x') =  w_R (|x'|) \qtext{for all} x_1\in \R, \, x'\in\R^{N-1}
\]
which proves the claim.
\end{proof}
Next, we prove that $u_\infty$ is bounded from below by a positive number up to the boundary of
$\R\times B_R'$. This is stated in the following lemma.
\begin{Lemma}\label{uwdelta}
There exists $\delta >0$ such that
\be
u_\infty(x_1,x')\geq w_R(|x'|)+\delta \quad \text{for } x\in \R\times B_R'.
\ee
\end{Lemma}
\begin{proof}
Assume by contradiction that for all $n\in \N^*$ there exists $x_n=(x_{1,n},x'_n)\in\R\times B_R'$ such that 
$$
u_\infty(x_{1,n},x'_n)\leq w_R(|x'_n|)+\frac 1n.
$$
Consider $v_n(x)=u_\infty(x_{1,n}+x_1, x')-w_R(|x'|)$ a function defined on $\R\times B_R'$. Then, up to extraction, we can assume that $x_n' \to x_0'$,  $v_n\to v$ in $C^2_{loc}(\R\times \overline{B_R'})$ and $v$ verifies
\begin{gather*}
v\geq 0, \quad v(0, x_0')=0, \quad v\not \equiv 0,\\
-\Delta v=c(x)v \text{ with } c \text{ bounded,}\\
v>0\quad \text{in } \partial B_R'.
\end{gather*} 
The last positivity property is obtained from $u_\infty >0$ in $\Omega$ and by using that $u_\infty$ satisfies the Neuman boundary condition, and passing to the limit in it, if need be. This contradicts the strong maximum principle.
\end{proof}
As a conclusion, $\displaystyle \inf_{\R \times B'_{R}} u_\infty\geq \delta$ and
%This estimate clearly yields Theorem~\ref{Thmpartial}. Indeed, assume that $\Omega$ contains the cylinder $\R \times B_R'$ for any $R > R_0$. Then, let $R_0 < R$. From the above inequality, it follows that
%\[
%\inf_{x_1\in\R, |x'|\leq R_0} u_\infty (x_1, x') \geq w_R(R_0) >0.
%\]
the proof of Theorem~\ref{Thmpartial} is thereby complete. \qed
\begin{Rk} Note that, Lemma \ref{uwdelta} actually proves a somewhat stronger result than the one stated in Theorem \ref{Thmpartial}\end{Rk}

\bigskip

In order to conclude that the state $u\equiv 1$ invades the whole domain, one has to require further assumptions on the  geometry of $\Omega$ . Indeed in domains as illustrated in Figure \ref{proppartial} there is propagation in the sense given above but the propagation could be blocked when entering some parts of the domain. One can find in \cite{BHM} or \cite{B} arguments leading to partial propagation in exterior domains and examples of domains where the propagation is actually only partial.
\begin{figure}[h!]
\begin{center}
\includegraphics[scale=0.75]{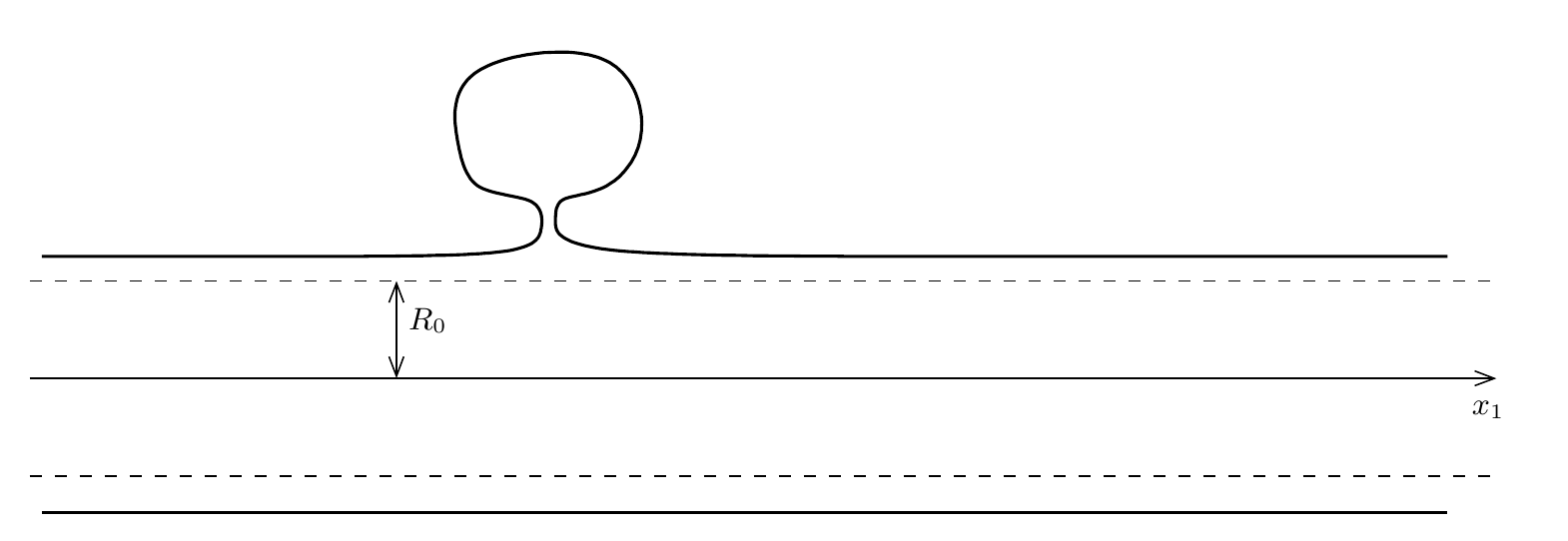}
\caption{\footnotesize{Example of domain in dimension 2 where propagation could be blocked in the small area on the top of $\Omega$.}}\label{proppartial}
\end{center}
\end{figure}

We now turn to this question of finding conditions under which the invasion by 1 is complete. 
%We are therefore interested in the solutions of  the limiting stationary problem
%In this section we establish sufficient conditions on the domain $\Omega$ that yield complete %propagation of 1 in $\Omega$.
Since $u(t,\cdot)\to u_\infty$ locally uniformly in $\Omega$ as $t\to \infty$, the limit $u_\infty$ is a solution of the following stationary problem
\begin{equation} \label{stateq}
\begin{cases}
-\Delta u_\infty=f(u_\infty) &\text{in }\; \Omega,\\
\partial_\nu u_\infty=0 &\text{on }Ê\; \partial \Omega,\\
u_\infty(x)\to1 &\text{as } \; x_1\to-\infty,
\end{cases}
\end{equation}
We are thus led to finding geometric assumptions on $\Omega$ that imply that $u_\infty\equiv1$. 

It is interesting to note that this is related to a classical result of Matano  \cite{M}  and Casten and Holland \cite{CH}  regarding the following non-linear elliptic Neumann boundary value problems in {\em bounded} domains $D$:
\begin{equation} \label{MCH}
\begin{cases}
-\Delta u=f(u) & \text{in } D,\\
\partial_\nu u=0 & \text{on } \partial D.
\end{cases}
\end{equation}
These authors have shown that if $D$ is a {\em convex and bounded} smooth domain, then the only stable solutions are constants. It is a natural question to ask whether there are extensions of this result to unbounded domains. What further conditions imposed on the solution and which geometrical conditions on the unbounded domain $D$ garantee that the only 
 stable solutions are constants? Our results on complete propagation are partial answers to this question for the solution $u_\infty$. More general formulations are an interesting open problem.
 %\subsection{Partial propagation}\label{partialprop}

%\begin{equation}\begin{cases} \label{stateq}
%-\Delta u =f(u) &\text{in } \Omega,\\
%\partial_\nu u=0 &\text{on } \partial\Omega,\\
%u(x_1,x')\to 1 &\text{when } (x_1,x')\in\Omega,\:x_1\to-\infty.
%\end{cases}
%\end{equation}
%This problem has been studied by Casten and Holland in \cite{CH} and Matano in \cite{M} for
%bounded domains. We are interested in these same questions on the geometric conditions that %lead non constant solution to be unstable but in unbounded domains.

We have shown in Theorem~\ref{IncreaseThm} that there exists a stable non constant 
solution of the previous stationary problem for domain containing a transition area of small enough measure. We have also shown in Theorem~\ref{decreaseThm} that $u_\infty \equiv 1$ for propagation in the direction of a decreasing domain. In the next sections we derive two different types of geometrical conditions that garantee the complete invasion by 1.

%%%%%%%%%%%%%%%%%%
%%%%%%%%%%%%%%%%%%
%%% Fin relecture 05/01 ˆ 00h35
%%%%%%%%%%%%%%%%%%
%%%%%%%%%%%%%%%%%%%%%%%%%%%%%%%%%%%%%%%%%%%%
\section{Complete propagation for  domains that are ``star-shaped" with respect to an axis}
% the sliding cylinder method}
\label{completeprop}
%%%%%%%%%%%%%%%%%%%%%%%%%%%%%%%%%%%%%%%%%%%%%%
The object of this section is to prove the result of complete invasion of Theorem \ref{Thmslidingcylinder}, that is $u_\infty\equiv 1$ in $\O$.
We consider a domain $\O$ such that $\O\supset \R \times B'_{R}$, with $R>R_0$, $R_0>0$ defined in section \ref{propgeneraldom}. Moreover we assume that at each point on the boundary $x=(x_1,x')\in\pO$, the outward unit normal $\nu$ makes a non-negative angle with the direction $x'$. To prove the result, we introduce here a new method which rests on sliding cylinders.

As in section \ref{propgeneraldom}, for all $R>R_0$, we denote $z=z_R$ the maximal positive solution of 
$$
\begin{cases}
 -\Delta z=f(z) &\text{in } B_{R},\\
 z=0 &\text{on } \partial B_{R}.
 \end{cases}
$$
It satisfies $0<z_R<1$ in $B_R$ and is radially symmetric and decreasing with respect to $|x|$. We extend $z_R$ by 0 outside $B_R$. We denote $z_R(x)= w_R(|x|)$ where $w_R$ is defined on $\R^+$. In particular, we know that $\sup z_R= z_R(0) = w_R(0) >\theta$.

By shifting the axis of the cylinder in orthogonal directions and rotating the cylinder, we will now construct from $w_R$ a new axisymmetric function that is a subsolution. %The coordinates of a point $x\in\R^N$ are denoted 
%$x:= (x_1, y)$ where $y$ stands for $y= (x_2, \ldots, x_N)$.
Let $h\geq 0$ a given real and $e\in \Sm$ a given unit vector defining a direction in $\Rm$.
We define, for all $x'\in\Rm$:
\[
\psihe (x') : = w_R (|x' - h e|)
\]
so that $\psihe (x') >0$ on $B'_R (he):= B'_R + he$, and
\be\label{subsol}
\ph (x') := \max_{e\in \Sm} \psihe (x').
\ee

Note that, since $w_R(|x'|)$ is axisymmetric and decreasing, $\ph(x')=w_R(|x'-h\frac{x'}{|x'|}|)$ for $x'\neq 0$. Hence $\ph$ is also  axisymmetric but not decreasing away from the origin anymore. Moreover $\ph(x')>0$ if and only if $h-R<|x'|<h+R$. Finally since $\ph$ is a supremum of a family of subsolutions, it is a  generalised subsolution, as defined in \cite{BHbook}. But we will only use classical subsolutions.
We consider it to be defined on all of $\R^N$, so that it is independent of the variable $x_1$.

By Lemma \ref{uwdelta}, we know that $u_\infty (x)> w_R(|x'|) \equiv \phi_0(x')$ for all $x\in\O$. 
We claim that $u_\infty(x) \geq \phi_{h}(x') $ for all $x\in \Omega$ and all $h\geq 0$. Hence $u_\infty(x)\geq \phi_{|x'|}(x')=w_R(0)> \theta$ and we will prove later on that this implies $u_\infty\equiv 1$.

So let us assume by contradiction that there exist $h_0>0$ and $x_0\in \Omega$ such that $u_\infty(x_0) < \phi_{h}(x'_0)$. We define
\be \label{he}
\he= \sup \{ h \geq 0\, , \; \forall h'\in [0,h] \;  \forall x\in\O  \;\; u_\infty(x) \geq \phi_{h'}(x')  \}
\ee
Clearly, $\he\geq 0 $ is well defined and finite. By continuity, we know that
$u_\infty(x) \geq \pe (y)$ for all $x\in\O$.

By definition of $\he$, there exists a sequence $h_n\searrow\he$ and a sequence of point 
$x_n\in\O$ such that $ u_\infty(x_n) < \phi_{h_n} (x'_n)$ where $x_n= (x_{1,n}, x'_n)$. Without loss of generality, we can assume that the sequence is bounded. Indeed, if the sequence $(x_n)_n$ is not bounded, then we can shift the origin by $(x_n)_n$  so that now $x_{1,n}=0$ but the domain and the function change with $n$. We can pass to the limit on these by a standard compactness argument.

Therefore, we can assume that the sequence $(x_n)_n$ converges to some point $p=(0, p')\in\overline{\O}$. In the limit, we get
$u_\infty(p ) \leq \pe(p')$ and since $u_\infty\geq\pe$ everywhere, we get $u_\infty(p') = \pe(p')$. First, we know from Lemma \ref{uwdelta}, that there exists $\delta>0$ such that
$u_\infty(x_1,0) \geq \max \phi + \delta$ for all $x_1\in\R$, so that it is impossible that $p'=0$.
It is also impossible that $p\in \O$. Indeed, there exists some $e\in\Sm$ such that $u_\infty(p') = \pe(p')= \psi_{h^\star,e} (p')>0$ and
$u_\infty\geq \psi_{h^\star,e}$ everywhere. In view of the strong maximum principle (recall that $u_\infty-\psi_{h^\star,e}$ satisfies some linear elliptic equation, since $u_\infty$ and $\psi_{h^\star,e}$ satisfy the same semi-linear equation), we infer that $u_\infty\equiv\psi_{h^\star,e}$. But this is impossible as $\psi_{h^\star,e}$ vanishes somewhere.

It remains to consider the case that $p\in\pO$. Since $u_\infty>0$, we know that 
$\he -R < |p'| < \he +R$. Let us show that, necessarily, $\he \leq |p'| < \he +R$. Indeed, suppose to the contrary that $\he -R < |p'| < \he $. Recall that $p' \not=0$. In the region
$\he -R < r <\he$, the function $\ph $ is decreasing with $h$. Thus, for small enough $\he -h >0$ we see that $\ph (p') > \pe(p') = u_\infty ( p)$. But this is in contradiction with the definition of $\he$ that requires
$u_\infty\geq \ph$ for such an $h$. We have reached a contradiction and, therefore, we know that
$\he \leq |p'| < \he +R$. 
%Now, for the point $q$, there exists a certain maximizing direction $e$ such that $\psihe(q)= \ph(q)$. Since $V$ is spherically symmetric decreasing away from the center [Gidas, Ni, Nirenberg], it is necessary for $e$ to be maximizing that $e$ be aligned with the line from $0$ to $q$, so that the distance from $q$ to the center $\he e$ be minimized.\\

We have now reached a situation where $u_\infty\geq \pe$ everywhere, $u_\infty(p )= \pe( p')=w_R(|p'-h^\star \frac{p'}{|p'|}|)=w_R(||p'|-h^\star|)$.

We compute the outward normal derivative using the fact that $w_R$ is decreasing to get
\[
\pn \pe (p') = - |w'_R(|p'-h^\star \frac{p'}{|p'|}|)| \frac{p'}{|p'|} \cdot \nu \leq 0.
\]
The term $\frac{p'}{|p'|}$ comes from the derivative of $|p'-h^\star \frac{p'}{|p'|}|$ with respect to $p'$ and the information that $\he \leq |p'| < \he +R$. The last inequality follows from our assumption 
\[
 \nu \cdot \frac{p'}{|p'|} \geq 0 \qtext{on} \pO
 \]
 The proof is thereby complete.
 
Now we want to prove that $u_\infty\equiv 1$ in $\Omega$. Assuming that $u_\infty\not\equiv1$, there exists a sequence $(x_n)_n$ in $\Omega$ such that  $u_\infty(x_n)\to\underset{x\in\Omega}{\inf} u_\infty(x)\in [w_R(0),1)$ as $n\to+\infty$. We recall that $w_R(0)>\theta$.
%If $(x_n)$ is bounded, up to extraction $x_n\to\bar{x}\in\Omega$ and using elliptic estimates we have that 
%$$-\Delta u_\infty(\bar{x})=f(u_\infty(\bar{x}))>0,$$ which is impossible.\\
%Otherwise if $(x_n)$ is not bounded, 
Letting 
$u_n(x):=u_\infty(x_n+x)$ and $D_n=\{x\in\R^N, \; x+x_n\in \Omega\}$, 
we know that $u_n\to \overline{u}$ and $D_n\to D$ as $n\to+\infty$ such that
$$
\begin{cases}
-\Delta\overline{u}=f(\overline{u})  &\text{in }D, \\
\partial_\nu \overline{u}=0 & \text{on } \partial D. 
\end{cases}
$$
If $0\in \partial D$, it contradicts Hopf's Lemma and if $0\in D$, then  
$\Delta\overline{u}(0)\leq 0$, but $f(\overline{u}(0))>0$
which is impossible.
\qed

\begin{Rks}
\begin{itemize}
\item
The assumption \pref{normal} is a star-shaped type property. It essentially says that any section in planes orthogonal to the $x_1$-- axis is star shaped with respect to the trace of the axis in this plane. Hence the name ``star shaped with respect to an axis'' 
\item
We can also treat the case when ``$\O$ ends to the right''. Namely $\O$ is a straight cylinder as $x_1\to -\infty$, and say is bounded to the right.
\end{itemize}
\end{Rks}
\begin{figure}[h!]
\begin{center}
\includegraphics[scale=0.4]{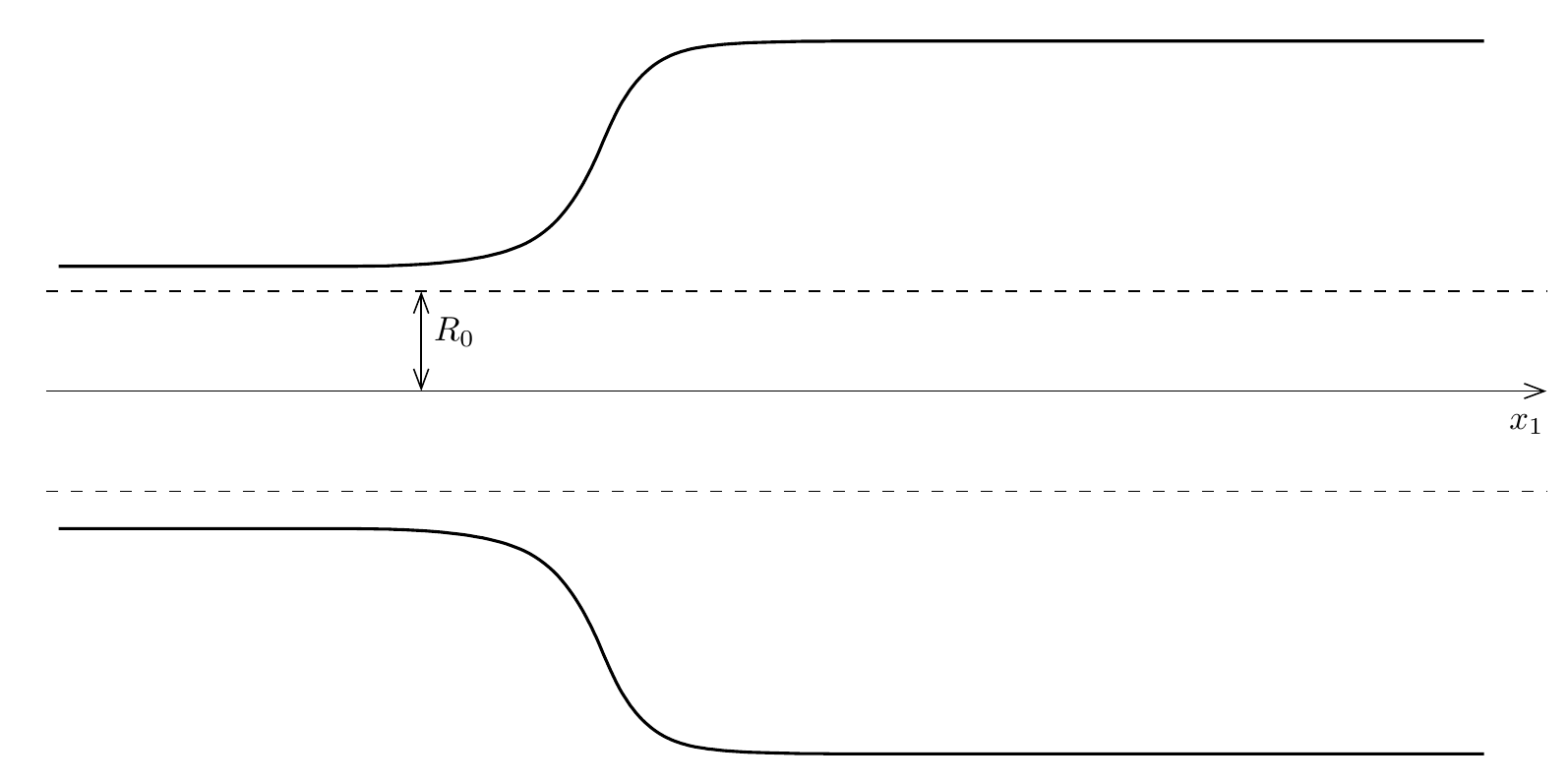}\hfill
\includegraphics[scale=0.5]{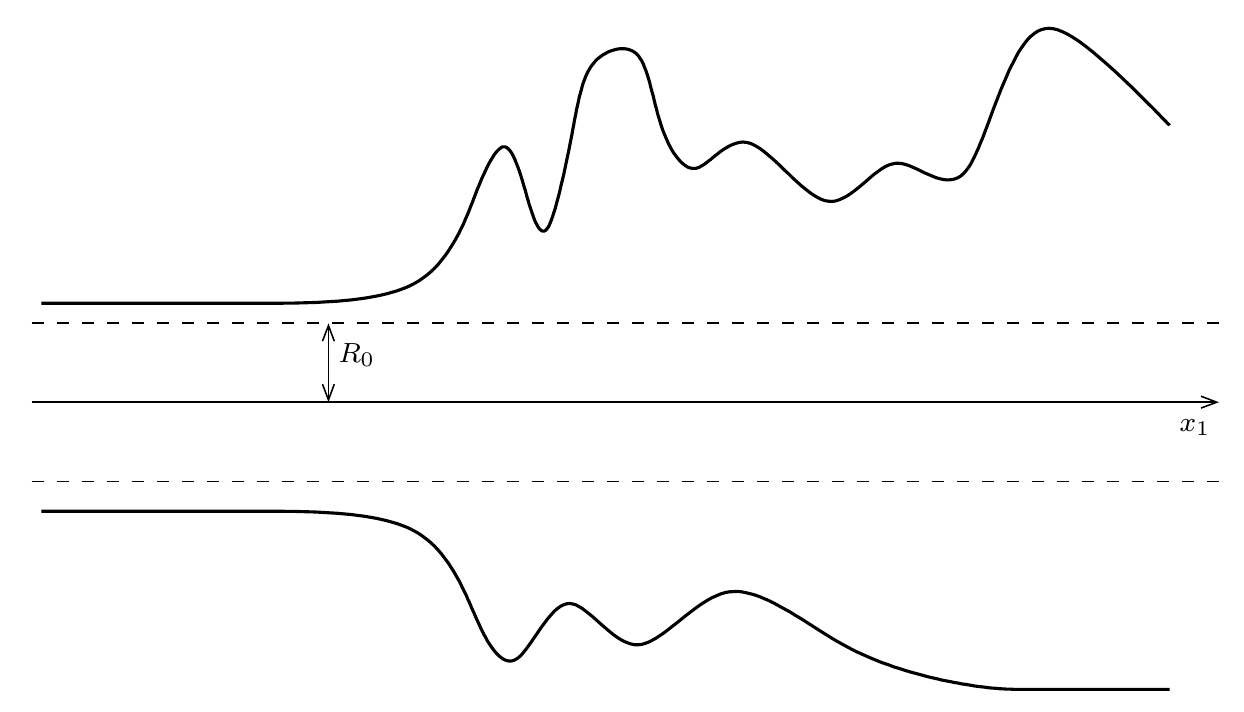}
\caption{\footnotesize{Examples of domains that satisfy the sliding cylinder assumption in dimension 2.}}\label{slidcylfigure}
\end{center}
\end{figure}
\begin{figure}[h!]
\begin{center}
\includegraphics[scale=0.4]{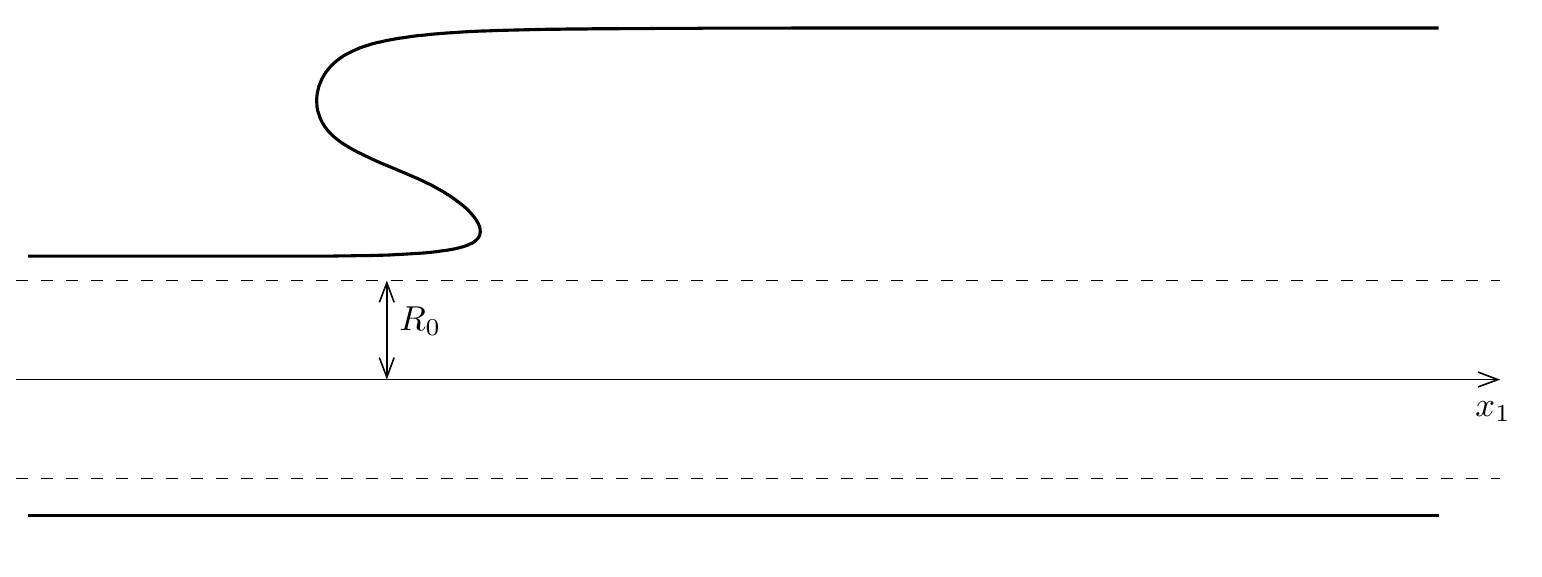}\hfill
\includegraphics[scale=0.4]{Omeganocyl2}
\caption{\footnotesize{Examples of domains that do not satisfy the sliding cylinder assumption in dimension 2.}}\label{noslidcylfigure}
\end{center}
\end{figure}
This proposition proves the complete propagation of 1 in a large variety of domains $\Omega$ but examples of domains that are not covered by this theorem are presented in figure \ref{noslidcylfigure}. Some domains, as the figure on the left of Figure \ref{noslidcylfigure}, that do not satisfy starshape assumption with respect to an axis  nevertheless satisfy the conditions  of Theorem \ref{increaseThmprop} and thus satisfy the complete propagation property as we will now prove in the next section.

%%%%%%%%%%%%%%%%%%%%%%%%%%%%%%%%%%%%%%%%%%%%%%%%%%%%%%%%%%
\section{Complete propagation in the direction of an increasing domain} \label{completepropwidening}
In this section we prove Theorem \ref{increaseThmprop}. 
Thus we assume that $\Omega$ satisfies the following assumptions:
\begin{subequations}\label{Omegawideningcompleteprop}
\be%\tag{Wa}
\R\times B'_{R}\subset\Omega
\ee
where $R> R_1$ and $R_1$ is defined below. 
We further assume that there exists $L>0$ such that ,
\be%\tag{Wb}
\Omega_L^r:=\left\{(x_1,x')\in\Omega,\: x_1>L\right\} \text{ is convex,}
\ee 
and that there exists $C>0$ such that 
\be
\{x\in \Omega, \; x_1<L+R\}\subset (-\infty, L+R) \times B_C,
\ee 
and \be%\tag{Wc}
\nu_1(x)\leq 0 \text{ for all } x\in\partial\Omega \text{ with } x_1<L+R 
.
\ee
\end{subequations}
\begin{figure}[h!]
\begin{center}
\includegraphics[]{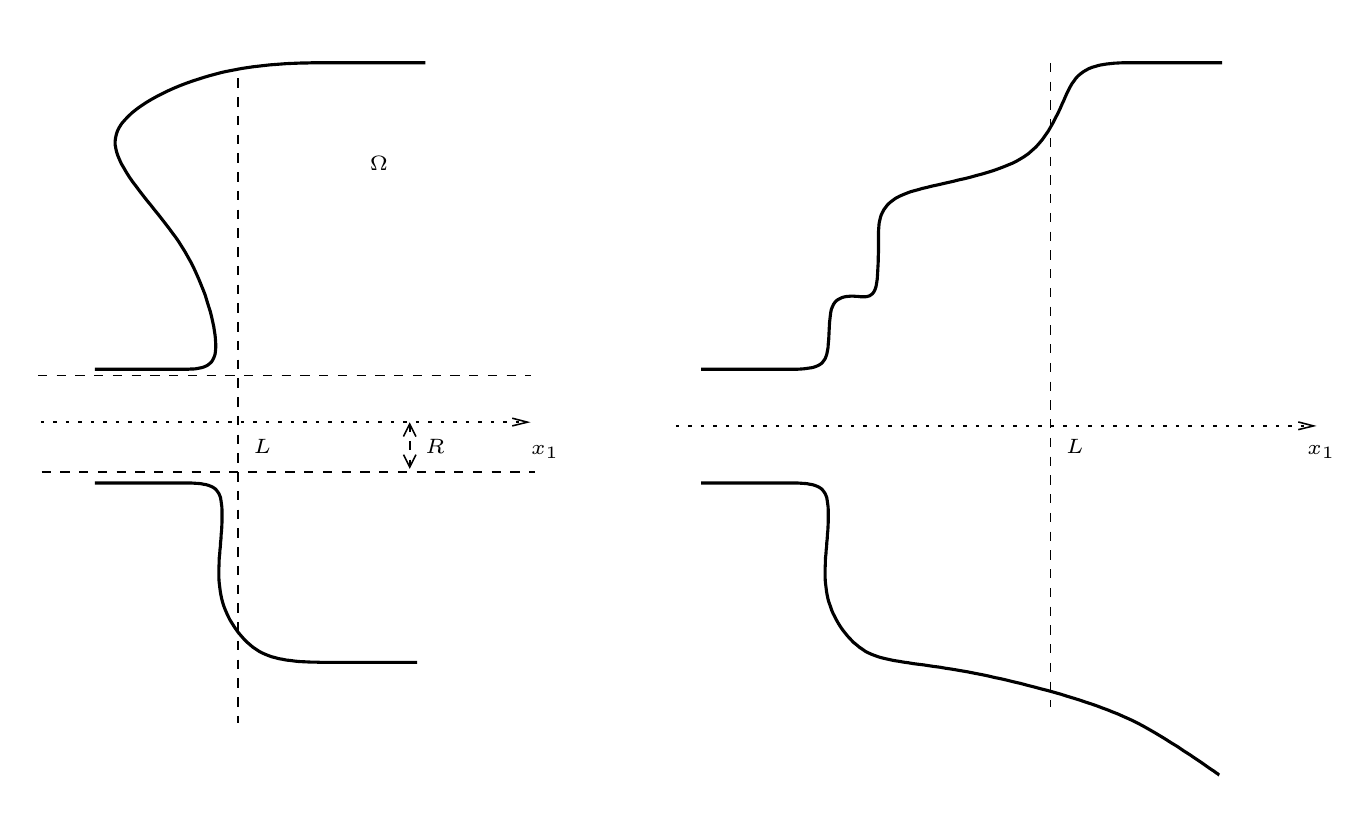}
\caption{\footnotesize{Domains $\Omega$ that satisfy \eqref{Omegawideningcompleteprop}}}\label{domboulefig}
\end{center}
\end{figure}
%  Let us remind that we know from Theorem \ref{existenceThm} that the unique solution of our parabolic problem is increasing in time. This implies, using parabolic estimates, that 
% $$u(t,x)\to u_\infty(x)\text{ as } t\to+\infty, \text{ locally uniformlay in } x\in\Omega$$
% and $u_\infty$ is a solution of 
% \be\label{pbstatdomR}
% \begin{cases}
% -\Delta u_\infty=f(u_\infty) &\text{in } \Omega,\\
% \partial_\nu u_\infty=0 &\text{on } \partial\Omega.
% \end{cases}
% \ee
We will first prove that $u_\infty$ is close to 1 in $$\Omega_{L+R}^r:=\Omega\cap\left\{(x_1,x')\in\R\times\R^{N-1},\: x_1> L+R\right\}$$ using a sliding method.  Then we will extend this estimate to all of $\Omega$ using a backward travelling front as subsolution  and finally we will conclude that $u_\infty \equiv 1$.

%%%%%%%%%%%%%%%%%%%%%%%%%%%%%%%%%%%%%%%%%%%%%
%\subsection{The sliding ball method}

As in section \ref{propgeneraldom}, we denote $z=z_R$ the maximal positive solution of 
\be\label{Dirichlet}
\begin{cases}
 -\Delta z=f(z) &\text{in } B_{R},\\
 z=0 &\text{on } \partial B_{R},
 \end{cases}
\ee
We recall that it satisfies $0<z_R<1$ in $B_R$, and is radially symmetric and decreasing with respect to $|x|$. We extend $z_R$ by 0 outside $B_R$ and denote $z_R(x)= w_R(|x|)$ where $w_R$ is defined on $\R^+$. In particular, we know that $\sup z_R= z_R(0) = w_R(0) >\theta$. We further recall from Lemma~\ref{comppropcyllemma} that for all $x=(x_1,x')\in \Omega$, we have
$u_\infty(x)\geq w_R(|x'|)$.

By a result of \cite{BL} we know that this maximal solution $w_R$ converges to 1 locally uniformly as $R\to+\infty$. Thus, there exist $R_1>R_0$, where $R_0>0$ is defined in section \ref{propgeneraldom}, and $\delta>0$ such that $w_R(0)-2\delta>\beta$ for all $R>R_1$, with $\beta>\theta$ defined by \eqref{beta} in the Introduction. Let $\eta>0$ be such that $w_R(|x|)>w_R(0)-\delta$ for all $x\in B_\eta$. We will prove that $u_\infty>w_R(0)-\delta$ in $\Omega_{L+R}^r$ using the following definition.
\begin{Def}[Sliding ball assumption]
We say that $x\in \Omega$ verifies the sliding ball assumption if there exists $a\in B_\eta(x)$, $x^0\in\R$ and 
\begin{gather}
 \gamma:[0,1]\to\Omega \text{ continuous, such that } \gamma(0)=(x^0,0,...,0),\quad \gamma(1)=a,\notag\\
%\quad x\in B_{R-\eta}(x_1),\notag \\
\text{and} \quad \forall \:t\in[0,1)\;\; \forall \: b\in \overline{B_R}(\gamma(t))\cap\partial\Omega\quad (b-\gamma(t))\cdot \nu\geq 0 \tag{H2}\label{Hypchemin}
\end{gather}
where $\nu$ is the outward unit normal at $b\in\partial\Omega$. 
\end{Def}
Assumption \eqref{Hypchemin} means that there exists a continuous path $\gamma$ from $a$ a point near $x$ to $(x^0,0,...,0)$ a point of the $x_1$-axis, such that for all $x_t$ on this path $B_R(x_t)\cap\Omega$ is star shaped with respect to the center $x_t$. %One can look at Figure \ref{domboulefig} in section \ref{completeprop} for some illustrations.\\
\begin{Lemma}
For all $x\in \Omega$ that satisfies the sliding ball assumption, we have
$$
u_\infty(x)\geq w_R(0)-\delta.
$$ 
\end{Lemma}
\begin{proof}
The proof is the same as the one of Lemma \ref{comppropcyllemma} with the novelty here that we slide the maximal solution $w_R$ along the path between the $x_1$-axis and $a$ and compare it to the solution $u_\infty$.

So we consider $x\in \Omega$ that satisfies the sliding ball assumption and 
an associated path $\gamma$ from $(x^0,0)$ on the $x_1$-axis to $a\in B_\eta (x)$ satisfying the properties in the definition. Let $w_R^{\bar{x}}$ be the maximal solution of the Dirichlet boundary value problem \eqref{Dirichlet} in $B_R(\bar{x})$. We know from  Lemma \ref{comppropcyllemma} that 
$w_R^{(x0,0)}\leq u_\infty,$ so we can define 
$$
t^\star=\sup\{t\in[0,1], \; \forall \tilde t\in [0,t] \; w_R^{\gamma(t)}\leq u_\infty \}.
$$ 
We will prove by contradiction that $t^\star=1$. Otherwise we define $v=u_\infty-w_R^{\gamma(t^\star)}$ and once again we have
\begin{gather*}
-\Delta v=c(x)v \text{ in }\Omega \text{ with } c \text{ bounded}, \quad v\geq 0, \quad v\not \equiv 0, \\
\text{there exists } \hat x \in \overline{B_{R}(\gamma(t^\star))} \text{ such that } v(\hat x)=0.
\end{gather*}
By the strong maximum principle, it is impossible that $\hat x\in \overline{B_{R}(\gamma(t^\star))}\cap \Omega$ so $\hat x \in B_{R}(\gamma(t^\star))\cap \partial \Omega$ and from Hopf's Lemma, we have $ \partial_\nu v(\hat x)<0$. But
$$
\partial_\nu v(\hat x)= -\partial_\nu \big(w_R(|\hat x -\gamma(t^\star)|)\big)=|w_R'(|\hat x -\gamma(t^\star)|)|\frac{\hat x -\gamma(t^\star)}{|\hat x -\gamma(t^\star)|}\cdot\nu
$$
since $w_R$ is decreasing. But this contradicts the definition of the sliding ball assumption. 

We thus get that $u_\infty\geq w_R^a$ and since $x\in B_\eta(a)$, we have $u_\infty(x)\geq w_R(0)-\delta$ by definition of $\eta$ and $\delta$.
\end{proof}
We claim that if $x=(x_1,x')\in \Omega$ verifies $x_1\geq L+R$, then 
$$u_\infty(x)\geq w_R(0)-\delta.$$
Indeed let us choose $x\in \Omega$ with $x_1\geq L+R$. Thanks to the previous lemma, it is sufficient to prove that $x$ satisfies the sliding ball assumption. Since $\Omega_L^r$ is convex, as assumed in \eqref{Omegawideningcompleteprop}, we define $a=x$, $x_0=x_1$ and $\gamma(t)=(x_1,tx')\in\Omega_L^r$ for all $t\in[0,1]$.
Assume by contradiction that there exist $t\in [0,1)$ and $b \in \overline{B_R}(\gamma(t))\cap\partial\Omega$ such that $(b-\gamma(t))\cdot \nu< 0$. Then $b$ and $\gamma(t)\in \overline{\Omega_L^r}$ which is convex so for any $\epsilon\in [0,1]$, $(1-\epsilon) b+\epsilon \gamma(t)=b- \epsilon (b-\gamma(t))\in  \overline{\Omega_L^r}$ but this contradicts the definition of the outward unit normal, for $\epsilon$ small.
\begin{Rk}
Note that $u_\infty>w_R(0)-\delta>\theta$ in $\Omega^r_{L+R}$ implies that $u_\infty(x)\to1$ as $x_1\to+\infty$.
\end{Rk}
%%%%%%%%%%%%%%%%%%%%%%%%%%%%%%%%%%%%%%%%%%%%%%
%\subsection{Comparison with a particular subsolution}

Now we will prove that $u_\infty(x)\geq w_R(0)-2\delta>\theta$ for all $x\in\Omega$ and thus $u_\infty\equiv 1$ in $\Omega$ following the same proof than in the previous section. This ends the proof of Theorem \ref{increaseThmprop}.

Let us define $f_\delta\in C^{1,1}([0,w_R(0)-2\delta])$ as follows:
\be\label{fdelta}
f_\delta(s)=\begin{cases} f(s) &\forall s\in[0,\beta],\\
		0 & \text{for } s=w_R(0)-2\delta,
		\end{cases}\\
\ee
with $0<f_\delta\leq f$ in $(\theta,w_R(0)-2\delta)$. And consider $(\phi_\delta,c_\delta)$ the travelling wave solution between 0 and $w_R(0)-2\delta$, invading 0 from the right, i.e $(\phi_\delta,c_\delta)$ is solution of 
\be\label{TWdelta}\begin{cases}
-\phi_\delta''-c_\delta\phi_\delta'=f_\delta(\phi_\delta)\quad\text{in } \R,\\
\phi_\delta(-\infty)=0,\quad \phi_\delta(+\infty)=w_R(0)-2\delta, \quad c_\delta<0.
\end{cases}\ee
We have $u_\infty(x)\to 1$ as $x_1\to -\infty$ so we can fix $-M<0$ such that $u_\infty(x)\geq w_R(0)-\delta$ on $\Omega_{-M}^l:=\{x\in \Omega, \; x_1<-M\}$. Now since $u_\infty\geq w_R(0)-\delta>w_R(0)-2\delta\geq \phi_\delta$ also on $\{x\in \Omega, \; x_1>L+R\}$, we have $u_\infty\geq \epsilon$ on $\Omega$ for $\epsilon$ sufficiently small and thus we can translate $\phi_\delta$ on the right such that,
$$\forall x\in\Omega,\quad u_\infty(x)\geq\phi_\delta(x_1).$$

We can now apply a parabolic comparison principle to $u_\infty$ and $v(t,x)=\phi_\delta(x_1-ct)$ on $\R_+\times \{x\in \Omega, \; -M<x_1<L+R\}$. Indeed $f$ is locally Lipschitz and
\begin{equation*}\begin{cases}
\partial_tu_\infty-\Delta u_\infty=f(u_\infty)   &\text{ on } \R_+^*\times\left\{(x_1,x')\in\Omega,\: -M<x_1< L+R\right\},\\ \partial_tv_\delta-\Delta v_\delta\leq f(v_\delta) &\text{ on } \R_+^*\times\left\{(x_1,x')\in\Omega,\: -M<x_1< L+R\right\},\\
\partial_\nu u_\infty(t,x) =0 \geq \partial_\nu v(t,x) &\text{ for } t>0 \text{ and } x\in \partial \Omega \text{ with } -M<x_1< L+R,\\
u_\infty(t,x)\geq v(t,x), &\text{ for } t\geq0\text{ and } x\in \Omega  \text{ with }x_1= L+R \text{ or } x_1=-M,\\
u_\infty(0,x)\geq v(0,x), &\text{ for }  x\in \Omega \text{ with } -M< x_1< L+R.
\end{cases}\end{equation*}
This yields $u_\infty(x)\geq \phi( x_1-c_\delta t)$ for all $t\geq 0$ and $x\in \Omega$ with $-M\leq x_1\leq L+R$ and passing to the limit for $t\to \infty$, we obtain $u_\infty(x)\geq w_R(0)-2  \delta$ on $\{x\in \Omega, \; -M\leq x_1\leq L+R\}$ and thus on all $\Omega$.
% with 
% \begin{equation*}
% c(t,x)=\begin{cases}\frac{f(v)-f(w)}{v-w}, &\text{for }v(t,x)\neq w(t,x),\\
% 				f'(0), &\text{otherwise}
% 	\end{cases}
% \end{equation*}
% $c\in L^\infty$. Applying the weak maximum principle we know that 
% $$\forall t\geq0,\: x\in\Omega\cap\left\{(x_1,x')\in\R\times\R^{N-1},\: x_1\leq L+R\right\},\quad h(t,x)\leq0.$$
% This implies that $u(t,x)\geq\phi_\delta(x_1-c_\delta (t-T))$ for all $x\in \Omega$, $t>T$ and letting $t\to+\infty$ we get
% $$\forall x\in\Omega,\quad u_\infty(x)\geq 1-2\delta>\theta.$$

Finally we conclude that $u_\infty\equiv 1$ as in the previous section.

\begin{Rk}
This proof can be adapted for more general domains than the ones considered in Theorem \ref{increaseThmprop}. The clue is to prove that the sliding ball assumption can be verified for any point of the domain far to the right.
\end{Rk}
\section*{Acknowledgments}
The research leading to these results has received funding from the European Research Council
under the European Union's Seventh Framework Programme (FP/2007-2013) / ERC Grant
Agreement n. 321186 - ReaDi - ``Reaction-Diffusion Equations, Propagation and Modelling'' held by Henri Berestycki. 
This work was also partially supported by the French National Research Agency (ANR), within the project NONLOCAL ANR-14-CE25-0013. Henri Berestycki was also partly supported by an NSF FRG grant DMS-1065979. Part of this work was carried out while he was visiting the Department of Mathematics at the University of Chicago. Juliette Bouhours was supported by a PhD fellowship "Bourse hors DIM" of the "R\'egion Ile de France". 

\bibliography{bibliothese}
\bibliographystyle{plain}

\end{document}